\documentclass[11pt]{article}

\usepackage{amsthm}
\usepackage{mathtools}
\usepackage{amsfonts}

\usepackage{booktabs}
\usepackage[a4paper, total={8.5in, 11in},margin=1in]{geometry}
 
\usepackage{comment}
\usepackage[font=small,labelfont=bf]{caption}
\usepackage{thmtools}
\usepackage{thm-restate}

\usepackage{dsfont}
\usepackage{amssymb}
\usepackage[dvipsnames]{xcolor}
\usepackage[ruled,vlined]{algorithm2e}

\usepackage[colorlinks=true,linkcolor=Maroon,citecolor=Blue]{hyperref}
\usepackage{cleveref}
\DeclarePairedDelimiterX{\infdivx}[2]{(}{)}{%
  #1\;\delimsize\|\;#2%
}

\newtheorem{theorem}{Theorem}[section]
\newtheorem{definition}[theorem]{Definition}
\newtheorem{proposition}[theorem]{Proposition}

\newtheorem{lemma}[theorem]{Lemma}
\newtheorem{claim}[theorem]{Claim}
\newtheorem{fact}[theorem]{Fact}
\newtheorem{observation}[theorem]{Observation}
\newtheorem{assumption}[theorem]{Assumption}
\newtheorem{remark}[theorem]{Remark}

\DeclareMathOperator*{\argmin}{arg\,min}

\DeclareMathOperator{\define}{\overset{def}{=}}

\DeclareMathOperator{\prox}{prox}

\DeclareMathOperator{\gd}{\textsc{GD}}
\DeclareMathOperator{\ogd}{\textsc{OGD}}
\DeclareMathOperator{\hgd}{\textsc{HGD}}
\DeclareMathOperator{\pp}{\textsc{PP}}
\DeclareMathOperator{\gogd}{\textsc{GOGD}}
\DeclareMathOperator{\pegd}{\textsc{PEGD}}
\DeclareMathOperator{\rgd}{\textsc{RGD}}
\DeclareMathOperator{\pid}{\textsc{PID}}

\date{}                     

\author{
  Ioannis Anagnostides\\[-2mm]
  National Technical University of Athens\\[-2mm]
  \texttt{ioannis.anagnostides@gmail.com}
  \and
  Ioannis Panageas\\[-2mm]
  University of California, Irvine\\[-2mm]
  \texttt{ipanagea@ics.uci.edu}
}

\title{Frequency-Domain Representation of First-Order Methods: \\
    \Large{A Simple and Robust Framework of Analysis}
}

\begin{document}

\maketitle
\pagenumbering{gobble}

\begin{abstract}

Motivated by recent applications in min-max optimization, we employ tools from nonlinear control theory in order to analyze a class of ``historical'' gradient-based methods, for which the next step lies in the span of the previously observed gradients within a time horizon. Specifically, we leverage techniques developed by Hu and Lessard (2017) to build a frequency-domain framework which reduces the analysis of such methods to numerically-solvable algebraic tasks, establishing linear convergence under a class of strongly monotone and co-coercive operators.

On the applications' side, we focus on the Optimistic Gradient Descent (OGD) method, which augments the standard Gradient Descent with an additional past-gradient in the optimization step. The proposed framework leads to a simple and sharp analysis of OGD---and generalizations thereof---under a broad regime of parameters. Notably, this characterization directly extends under adversarial noise in the observed value of the gradient. Moreover, our frequency-domain framework provides an exact quantitative comparison between simultaneous and alternating updates of OGD. An interesting byproduct is that OGD---and variants thereof---is an instance of PID control, arguably one of the most influential algorithms of the last century; this observation sheds more light to the stabilizing properties of ``optimism''. 
\end{abstract}

\clearpage

\pagenumbering{arabic}
\section{Introduction}

Gradient-based algorithms have received extensive study in the optimization and machine learning communities due to their simplicity and their convergence properties \cite{DBLP:conf/colt/GeHJY15, DBLP:journals/mp/LeePPSJR19} (and references therein). Indeed, these methods have recently found numerous applications in non-convex optimization  (training of Neural Networks), as well as in min-max optimization (training of GANs \cite{DBLP:conf/nips/GoodfellowPMXWOCB14}). 

In min-max optimization, however, cycling behavior has been observed even in seemingly innocuous settings such as unconstrained bilinear landscapes, and several remedies have been suggested to ensure convergence \cite{DBLP:conf/iclr/DaskalakisISZ18, DBLP:conf/iclr/MertikopoulosLZ19, DBLP:conf/innovations/DaskalakisP19, DBLP:conf/iclr/WeiLZL21}. Specifically, some of the proposed variants use positive or negative momentum (e.g., Optimistic Gradient Descent/Ascent \cite{DBLP:conf/iclr/DaskalakisISZ18}). Momentum is a technique that gives different importance to the most recent values of the iterates and their gradients (depending on whether it is negative or positive) and can provably give better convergence guarantees. Nevertheless, dealing with the values of past iterates makes the analysis particularly challenging. Typical technical arguments include spectral analysis of the Jacobian of the corresponding updating rule \cite{pmlrpan,DBLP:conf/iclr/ZhangY20}---establishing only local convergence---or ad hoc potential function arguments (e.g., contraction of the KL divergence \cite{DBLP:conf/iclr/WeiLZL21}), and unfortunately the analysis works case by case. 
In this context, our work addresses the following question:
\smallskip

\begin{quote}
    \textit{Is there a simple and general framework that allows for global analysis of gradient-based methods wherein the update rule depends on the history---within a finite time horizon---of the past gradients?}
\end{quote}

Formally, for a finite time horizon $T\geq 1$, such ``historical'' methods can be described with the following dynamical system:

\begin{equation}\label{eq:history}
    x_{k+1} = x_k + \textrm{span}(F(x_k), \dots, F(x_{k-T+1})),
\end{equation}
where $F$ represents an operator associated with the gradients of the underlying objective function. For example, in the context of min-max optimization this operator would take the form $F^{\top} = [\nabla_x f(x,y)^{\top}, - \nabla_y f(x,y)^{\top}]$, for some objective function $f(x,y)$ \cite{diakonikolas2021potential}. Our primary contribution is to answer this question in the affirmative, developing a systematic, robust (and rather elementary) analysis of gradient-based methods that make use of past gradients \eqref{eq:history}. To be more precise, one of the key observations is that, assuming linear dynamics, the aforementioned system \eqref{eq:history} can be analyzed very cleanly through a frequency-domain representation of the dynamics derived via the $Z$-transform, the discrete-time analog of the Laplace transform. One of our main technical insights is to show that such clean characterizations can be extended well-beyond (bi)linear settings via tools from nonlinear control theory. More concretely, we derive a succinct representation of the so-called \emph{transfer function} of the controller associated with each optimization method (see \Cref{tab:transfer_function}), and then it suffices to characterize this rational function via well-known and elementary (at least for ``low-degree'' methods such as Optimistic Gradient Descent/Ascent) calculations. These techniques turn out to be different from just spectral analysis of the Jacobian \cite{pmlrpan, DBLP:conf/iclr/ZhangY20} because they allow for global convergence guarantees and tight rates. To the best of our knowledge such reductions have not been used before in the analysis of gradient-based methods.

\paragraph{Our Contributions.} Most existing results regarding the analysis of systems of the form \eqref{eq:history} suffer from at least one the following caveats: (i) the analysis is particularly complex and ad hoc in nature, (ii) the bounds obtained with respect to the learning rate and the rate of convergence are far from sharp, (iii) the scope of the analysis is fairly limited; for example, the convergence guarantee only applies for unconstrained and bilinear games. 

Our approach can effectively address all of these issues. More concretely, let us consider as a benchmark the Optimistic Gradient Descent method ($\ogd$). Firstly, our analysis is considerably simpler than the existing ones, reducing the characterization of $\ogd$ to performing elementary calculations. Secondly, our approach reveals the exact region of convergence for $\ogd$. Indeed, we prove that our (stability) bound with respect to the learning rate is existentially tight; we are not aware of any prior works establishing this result. Sharp bounds are also obtained for the rate of convergence. Thirdly, our framework applies under the broad class of strongly monotone and co-coercive operators (see \Cref{assumption:sector}), going well-beyond the usually considered (bi)linear dynamics. In fact, even for the standard (``single-player'') setting, \Cref{assumption:sector} is strictly more expressive than the usual smoothness and (strongly) convexity assumptions \cite{hu2017control}.

Our results are also particularly robust in several ways. First, we are able to analyze the so-called \emph{generalized} $\ogd$ method, previously considered by Mokhtari et al. \cite{DBLP:conf/aistats/MokhtariOP20}. Specifically, they analyzed an extension of $\ogd$ wherein the coefficients are slightly perturbed. In this work we manage to obtain a simple characterization of the generalized $\ogd$ method under a much broader regime of parameters, strengthening their results. Importantly, we are able to analyze methods well-beyond this slight extension of $\ogd$. In particular, building on the intuition of $\ogd$, it is natural to ask whether we can analyze algorithms which linearly combine multiple ``historical'' gradients, i.e. $T > 2$ in \eqref{eq:history}, instead of only the current and the previously observed gradients (as in $\ogd$). In this context, we explain how one can \emph{numerically} analyze such a broad class of algorithms using an appropriate reduction. In fact, this reduction goes beyond the class of methods captured in \eqref{eq:history}, allowing to linearly combine the previous states as well (see \Cref{remark:extension}). We stress that this general characterization appeared elusive with some of the techniques employed in prior works. Finally, our results directly extend in the presence of \emph{relative deterministic noise}, a standard model in control theory wherein an adversary can corrupt the observed value of the gradient depending on its current norm; cf., see \cite{DBLP:journals/corr/abs-2107-10211}.

An interesting byproduct of our results is a connection between the notion of optimism and well-established techniques in control theory. Specifically, arguably the most widely-employed algorithm in control theory is the so-called \emph{proportional-integral-derivative} ($\pid$) control. We make the following observation:

\begin{observation}
Optimistic Gradient Descent (and generalizations thereof) is an instance of $\pid$ control.
\end{observation}

This nexus sheds additional light to the stabilizing properties of Optimistic Gradient Descent/Ascent in the context of zero-sum games. That is, the addition of the ``optimistic'' term can be very well-understood from the viewpoint of control theory, and in fact, this opens the door to the employment of an immense amount of theoretical and empirical results regarding $\pid$ control to better understand $\ogd$. More broadly, we believe that this connection may encourage further interdisciplinary research between online learning and control theory. 

Finally, we provide a qualitative comparison between \emph{simultaneous} and \emph{alternating} Optimistic Gradient Descent/Ascent in bilinear games. This consideration is concretely motivated since rigorously understanding the differences between these dynamics (beyond $\ogd$) constitutes a major open problem in optimization. In this context, we derive an exact region of stability for the alternating $\ogd$ method, and we compare its rate of convergence to that of the simultaneous dynamics. The unifying thread with our other results is a frequency-domain framework of analysis, which effortlessly gives exact bounds with elementary and remarkably simple techniques. We should note that a similar characterization was established in \cite{DBLP:conf/iclr/ZhangY20} with different tools; see also \cite{DBLP:journals/corr/abs-2102-09468}.

\paragraph{Technical Overview.} Most of our results are established based on tools from nonlinear control theory. Specifically, we use the formulation devised by Hu and Lessard \cite{hu2017control}, which views first-order optimization methods in the context of the so-called \emph{Lur'e} problem, one of the cornerstones of nonlinear control theory. More precisely, in this formulation one has to analyze two separate components: $(1)$ the nonlinearity (usually referred to as the \emph{plant} in the literature of control theory), which corresponds to the operator (e.g. the min-max gradients), and $(2)$ the controller, which represents the optimization method employed. As in \cite{hu2017control}, we will employ the fundamental \emph{small gain theorem} \cite{Khalil:1173048}, wherein it suffices to bound the \emph{gain}---the maximum ratio of the norm of the output to the norm of the input---of each component.

As a result, establishing linear convergence for an optimization method requires the following simple steps. First, deriving a frequency-domain representation of each algorithm; in \Cref{tab:transfer_function} we illustrate this representation for the methods studied in the present work, leading to a particularly succinct representation of first-order algorithms. The second step requires evaluating the stability of the induced transfer function\footnote{To simplify our high-level description, we omit certain additional (simple) transformations required in these steps.} under the set of parameters associated with the optimization method. For ``low-degree'' methods---such as $\ogd$---this can be simply performed by analytically determining the poles (the roots of the denominator) of the induced transfer function; nonetheless, we note that even for arbitrary polynomials the stability can be efficiently tested using an array of standard schemes \cite{1457261,989164}. Finally, it suffices to maximize the transfer function over the unit circle in the $z$-plane in order to determine the gain of the controller, and verify that the conditions of the small gain theorem are met. Importantly, we stress that these steps can be numerically automated under a generic class of first-order methods.

We apply this framework to establish linear convergence for $\ogd$, as well as its generalization, leading to a considerably simpler analysis than the existing ones. Interestingly, we are also able to give a remarkably simple analysis for the Proximal Point method, which is also an instance of $\pid$ control. A particularly compelling feature of this framework is its modularity. That is, even if we impose different assumptions on the operator, the analysis for each controller remains intact, largely simplifying future extensions and applications. For example, we leverage this feature to obtain bounds under (a limited amount of) adversarial noise on the observed value of the operator. In particular, the addition of noise can be incorporated in the nonlinear component, and by virtue of our previous observation it suffices to derive the new gain of the noisy nonlinearity, without altering the analysis for the controller. It is also interesting to note that this frequency-domain representation allows for a visualization of each optimization method via the \emph{circle criterion} (sometimes called generalized Nyquist), illustrating the effect of ``optimism'' on the standard Gradient Descent method (see \Cref{fig:Nyquist-GOGD}).

\begin{table}[ht]
    \centering
    \begin{tabular}{ c|c|c } 
 \toprule
 Optimization Algorithm & Parameters & Transfer Function \\
 \midrule
 $\pid$ Control \eqref{eq:PID-control} & $(K_P, K_I)$ & $\frac{-(K_P + K_I)z + K_P}{z^2 - z}$ \\ \hline
 Gradient Descent ($\gd$) \eqref{eq:GD} & $\eta$ & $-\frac{\eta}{z-1} $ \\ \hline
 Optimistic $\gd$ \eqref{eq:OGD} & $\eta$ & $\frac{\eta(1-2z)}{z^2 - z}$ \\ \hline
 Gen. Optimistic $\gd$ \eqref{eq:GOGD} & $(\alpha, \beta)$ & $\frac{\beta - (\alpha + \beta)z}{z^2 - z}$ \\ \hline
 Proximal Point \eqref{eq:PP} & $\eta$ & $- \frac{\eta z}{z - 1}$ \\  \bottomrule
\end{tabular}
    \caption{The frequency-domain representation of the optimization algorithms analyzed in this work. We should point out that the general $\pid$ control is more expressive than the form illustrated in this table (refer to \Cref{proposition:PID-TF}), subsuming even the Proximal Point method.}
    \label{tab:transfer_function}
\end{table}

Overall, we argue that these techniques are particularly well-suited for the analysis of \emph{single-call} variants of the Extra-Gradient method \cite{WANG2001969}. Nevertheless, it seems unclear how to apply this framework to characterize algorithms such as the vanilla Extra-Gradient. In the context of min-max optimization, a technical difficulty arises when the dynamics are not symmetric with respect to the two players, as is the case in \emph{alternating} dynamics. However, we can still derive a frequency-domain representation of the dynamics under the usual hypothesis of a bilinear objective function. In particular, this representation follows directly by transferring the time-domain dynamics to the $z$-space. Then, the behavior of the dynamics is captured by the characteristic equation, and it boils down to analyzing the roots of a (low-degree) polynomial, which is straightforward.

\paragraph{An illustrative example.} To be more concrete, we explain how the proposed framework yields a sharp analysis of vanilla Gradient Descent via the small gain theorem (\Cref{theorem-sgt}), which is due to Hu and Lessard \cite{hu2017control}. In particular, their analysis consists of the following observations:

\begin{enumerate}
    \item The controller of $\gd$ is $K(z) = - \eta/(z-1)$, while for $\eta = 2/(L + \mu) = h^{-1}$ the \emph{complementary sensitivity function} $K'(z) \define K(z)/(1 - hK(z))$ can be expressed as $K'(z) = - \eta/z$;
    \item $K'(\rho z)$ is stable for any $\rho \in (0,1)$;
    \item For $|z| = 1, |K'(\rho z)| = \eta/\rho$.
\end{enumerate}

Thus, if we invoke the small gain theorem, and in particular \Cref{theorem:main_theorem}, we can conclude that $\gd$ exhibits linear convergence with rate $\rho$ for any $\rho$ such that $\eta/\rho < 2/(L - \mu) \iff  \rho > (L-\mu)/(L + \mu)$, which is a well-known result. This simple skeleton will be used for the analysis of more complicated methods, such as $\ogd$.

\paragraph{Related Work.} There has been a tremendous amount of research in recent years in the interface of optimization, game theory, and even control theory. Most notably, we emphasize on the following directions: 

\paragraph{Limit Cycles in Zero-Sum Games.} It is well-documented by now that extensive families of \emph{no-regret} algorithms, such as Follow the Regularized Leader (FTRL), exhibit limit cycles or recurrence behavior in zero-sum games and even potential games \cite{DBLP:conf/soda/MertikopoulosPP18,DBLP:conf/nips/Vlatakis-Gkaragkounis19a,DBLP:conf/colt/CheungP19,DBLP:conf/nips/PalaiopanosPP17,DBLP:conf/innovations/PapadimitriouP16}, although the \emph{time-average} of these (no-regret) dynamics asymptotically approaches an equilibrium of the game. Unfortunately, a regret-based analysis cannot distinguish between a self-stabilizing system and one with recurrent cycles, and as a result, inherently different techniques are required to establish \emph{last-iterate} convergence.

\paragraph{$\ogd$ and Optimism.} In light of the shortcomings of standard paradigms such as FTRL, a recent breakthrough was made by Daskalakis, Ilyas, Syrgkanis, and Zeng \cite{DBLP:conf/iclr/DaskalakisISZ18}, showing that a simple variant of Gradient Descent, namely Optimistic Gradient Descent ($\ogd$), exhibits pointwise convergence to an equilibrium of the game, assuming that the objective function is \emph{unconstrained} and \emph{bilinear}, i.e. $f : \mathbb{R}^n \times \mathbb{R}^m \ni (x,y) \mapsto x^{\top} \mathbf{A} y$, for a matrix $\mathbf{A} \in \mathbb{R}^{n \times m}$. Several aspects of their analysis were improved in follow-up work \cite{DBLP:conf/aistats/LiangS19,DBLP:conf/aistats/MokhtariOP20,DBLP:conf/iclr/MertikopoulosLZ19,DBLP:conf/iclr/ZhangY20}, with the state of the art characterization being due to Wei et al. \cite{DBLP:conf/iclr/WeiLZL21}. The last-iterate convergence of $\ogd$ is very much pertinent to a classical almost $50$-years old result due to Korpelevich \cite{Korpelevich1976TheEM}, showing that, unlike Gradient Descent, the Extra-Gradient method (pointwise) converges to a solution of a certain (constrained) Variational Inequality (VI) problem without requiring \emph{strict} monotonicity---as is the case for bilinear games. The connection is that $\ogd$ constitutes a so-called \emph{single-call} variant of the Extra-Gradient method, as it essentially requires a single gradient evaluation per step. 

More broadly, $\ogd$ is based on the technique of \emph{optimism}, which has led to accelerated online learning algorithms under \emph{predictable} sequence patterns \cite{DBLP:conf/colt/RakhlinS13,DBLP:conf/nips/RakhlinS13,DBLP:conf/nips/SyrgkanisALS15}. As a result, optimism has attained two fundamental desiderata: \emph{stability} and \emph{acceleration}, leading to a substantial body of work in the last few years from various perspectives; cf.,  \cite{DBLP:conf/nips/DaskalakisP18,DBLP:conf/nips/CheungP20}. Beyond Optimistic Gradient Descent, there has also been interest in understanding the convergence properties of the Optimistic Multiplicative Weights Update (OMWU) method, commencing from Daskalakis and Panageas \cite{DBLP:conf/innovations/DaskalakisP19}; see also \cite{DBLP:journals/corr/abs-2107-01906,DBLP:conf/iclr/WeiLZL21} for several extensions and improvements in the analysis. 

\paragraph{Other Approaches.} In the context of training GANs, many other approaches have been proposed to address the shortcomings of traditional Gradient Descent. Most notably, Chavdarova et al. \cite{DBLP:conf/iclr/ChavdarovaPSFJ21} used the technique of \emph{lookahead} to stabilize training, previously proposed in ``single-agent'' optimzation by Wang et al. \cite{DBLP:conf/icassp/WangTBR20}. Moreover, Balduzzi et al. \cite{DBLP:conf/icml/BalduzziRMFTG18} proposed the Sympletic Gradient Adjustment algorithm, which is based on a decomposition of the dynamics to a \emph{potential} game and a \emph{Hamiltonian} game. Closely related to the intuition of $\ogd$, Gidel et al. \cite{DBLP:conf/aistats/GidelHPPHLM19} employed \emph{negative momentum} to stabilize the dynamics; see also the work of Zhang and Wang \cite{DBLP:conf/aistats/Zhang021} for explicit rates of convergence using negative momentum beyond bilinear games. Furthermore, Wang et al. \cite{DBLP:conf/iclr/WangZB20} proposed the \emph{Follow-the-Ridge} algorithm, and they showed that it always converges to \emph{local minimax} points. Finally, techniques from control theory have also been used in this line of work; see  \cite{DBLP:journals/corr/abs-1909-13188,DBLP:conf/icml/CheungP21}, and references therein. We should note that many other empirically-based approaches to obtain stability---without theoretical guarantees---have been proposed; most notably, we refer the interested the reader to the work of Salimans et al. \cite{DBLP:conf/nips/SalimansGZCRCC16}.

\paragraph{Alternating vs Simultaneous Updates.} Most of the results we have stated thus far are applicable for the so-called \emph{simultaneous} updates. Nonetheless, \emph{alternating} dynamics are also of particular importance in optimization and game theory; we refer to the excellent account in \cite{pmlr-v125-bailey20a}. Specifically, one notable applications relates to understanding the remarkable success of $\textsc{CFR}+$ \cite{10.1145/3131284} (in games such as poker), a variant of the widely-studied \emph{counterfactual regret minimization} ($\textsc{CFR}$) scheme, which switches from simultaneous to alternating updates. 

\paragraph{$\pid$ Control.} The so-called \emph{proportional-integral-derivative} ($\pid$) control constitutes one of the cornerstones of control theory, and it remains the main paradigm employed in modern industrial applications. There is voluminous amount of research regarding $\pid$ control; we refer to some of the standard textbooks/surveys on the matter \cite{johnson2005pid,knospe2006pid,visioli2006practical,aastrom2006advanced,li2006pid}. More recently, there has also been interest in employing $\pid$ control even in deep learning applications, as many commonly used optimization methods are essentially subsumed by the $\pid$ controller \cite{DBLP:conf/cvpr/AnWSXDZ18}.

\paragraph{Automated Analysis of Algorithms.} The idea of performing automated analysis of algorithms has also appeared in the literature of optimization, using \emph{semi-definite programming} (SDP) in order to solve \emph{linear matrix inequalities} (LMIs); for detailed discussions on this matter we suggest the excellent work of Lessard et al. \cite{DBLP:journals/siamjo/LessardRP16}. We also refer to the more recent work of Zhang et al. \cite{DBLP:journals/corr/abs-2009-11359}.

\section{Preliminaries}
\label{section:preliminaries}

Before we proceed with some important concepts from optimization, let us first introduce some basic notation. In particular, for a vector $x \in \mathbb{R}^d$, we will use $||x||$ to represent its Euclidean norm. For vectors $x, y$ we will write $\langle x, y \rangle$ to denote the inner product of $x$ and $y$. In the sequel, $F$ will represent a \emph{single-valued} operator, such that $F: \mathbb{R}^d \to \mathbb{R}^d$ for some $d \in \mathbb{N}$. 

\begin{definition}[Monotonicity]
The operator $F$ is said to be $\mu$-monotone, for some $\mu \geq 0$, if for all $x, x' \in \mathbb{R}^d$,

\begin{equation}
    \langle F(x) - F(x'), x - x' \rangle \geq \mu || x - x' ||^2.
\end{equation}
For $\mu = 0$, the operator will be referred to as monotone.
\end{definition}

\begin{definition}[Co-coercivity]
The operator $F$ is said to be co-coercive with parameter $1/L$, for some $L > 0$, if for all $x, x' \in \mathbb{R}^d$,

\begin{equation}
    \langle F(x) - F(x'), x - x' \rangle \geq \frac{1}{L} || F(x) - F(x') ||^2.
\end{equation}
\end{definition}

It should be noted that the Cauchy-Schwarz inequality implies that a co-coercive operator with parameter $1/L$ is also $L$-Lipschitz. We will make the following assumptions for the operator $F$:

\begin{assumption}
    \label{assumption:sector}
    The operator $F$ is $\mu$-monotone and co-coercive with parameter $1/L$.
\end{assumption}

With a slight abuse of notation, the ratio $\kappa := L/\mu > 1$ will be referred to as the \emph{condition number} of the operator $F$.

\begin{assumption}
    \label{assumption:unique}
    There exists a unique point $x^* \in \mathbb{R}^d$ such that $F(x^*) = 0$.
\end{assumption}

\paragraph{Min-Max Optimization.} In the context of min-max optimization, let $f(x,y) : \mathbb{R}^n \times \mathbb{R}^m \ni (x,y) \mapsto \mathbb{R}$ represent the \emph{objective function} of the game, such that player $x$ represents the ``minimizer'' and player $y$ the ``maximizer''. The operator $F$ of the \emph{min-max gradients} is defined as $F^{\top} = [ \nabla_x f(x,y)^{\top}, - \nabla_y f(x,y)^{\top} ]$. It is well-known that if $f(x,y)$ is $\mu$-strongly convex with respect to $x$ and $\mu$-strongly concave with respect to $y$ (which is the usual assumption in the literature), then $F$ is $\mu$-monotone; for completeness, a proof of this statement is included in \Cref{appendix:monotonicity}. The co-coercivity assumption for the min-max gradients is also fairly standard in the literature of min-max optimization; cf. \cite{diakonikolas2021potential}. It should be noted that co-coercivity implies monotonicity and Lipschitz continuity, but the opposite is not necessarily true, unless, for example, $F$ represents the gradient of a smooth and convex function. Nonetheless, the extension to general smooth min-max optimization is possible via \emph{approximate resolvent operators} \cite{DBLP:conf/colt/Diakonikolas20,diakonikolas2021potential}, and hence, co-coercivity captures the setting of smooth min-max optimization. Finally, although our emphasis lies on the unconstrained setting, extensions to constrained games are possible via \emph{operator mappings} \cite{DBLP:conf/colt/Diakonikolas20}.

\subsection{Tools from Nonlinear Control Theory} 

Throughout this work we will use the formulation devised by Hu and Lessard \cite{hu2017control}, which views first-order methods as a \emph{feedback interconnection system}; see \Cref{fig:P-K}. Specifically, this reduces the analysis to characterizing two separate components:

\begin{itemize}
\item The (static) nonlinearity $P$ associated with the operator $F$;
\item The controller $K$ associated with the optimization method.
\end{itemize}

\begin{figure}
\centering
\includegraphics[scale=0.7]{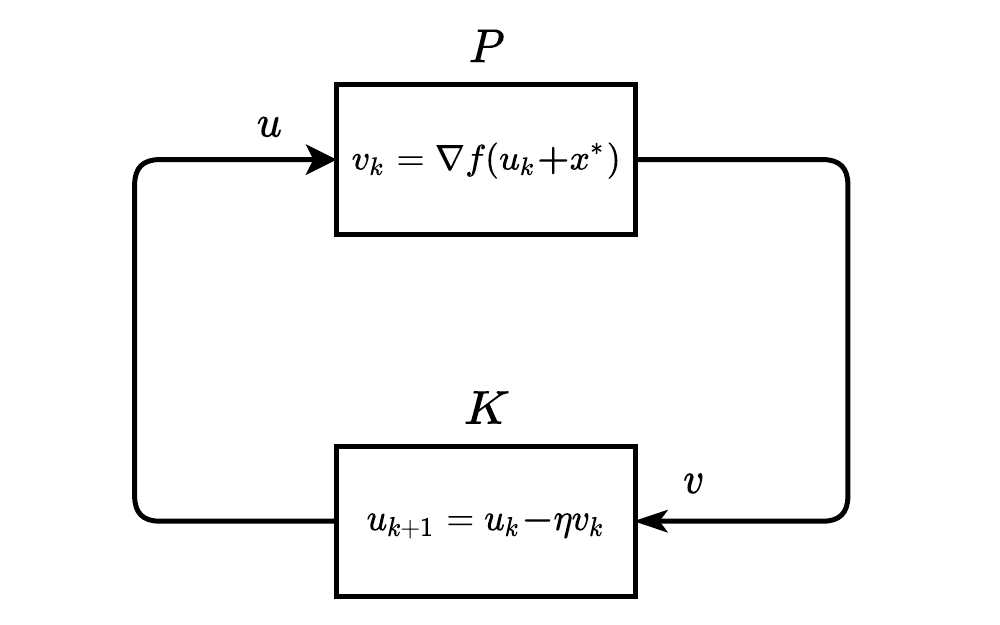}
\caption{Casting Gradient Descent as a $[P, K]$ feedback interconnection loop.}
\label{fig:P-K}
\end{figure}

In particular, Hu and Lessard \cite{hu2017control} used the fundamental \emph{small gain theorem} to analyze the induced feedback interconnection, and subsequently to characterize algorithms such as Gradient Descent. Before we proceed with their main technical result, let us first introduce some basic concepts from control theory. First, the (bilateral) $Z$-transform of a signal $u_k$ is defined as $U(z) = \mathcal{Z}\{u_k\} = \sum_{-\infty}^{+\infty} u_k z^{-k}$. We will extensively use the \emph{time-delay} property of the $Z$-transform: $\mathcal{Z} \{ u_{k - r} \} = z^{-r} U(z)$. A \emph{linear time-invariant} system (LTI) can be described through its transfer matrix $\mathbf{K}(z)$, such that for any input/output pair $(u, v)$ it holds that $V(z) = \mathbf{K}(z) U(z)$. The transfer matrix $\mathbf{K}(z)$ is said to be \emph{strictly proper} if $\mathbf{K}(\infty) = 0$; this property ensures that the feedback interconnection loop of \Cref{fig:P-K} is \emph{well-posed}. Moreover, the \emph{gain} $||H||$ of a system $H$ is defined as the minimal (finite) $\gamma \geq 0$ such that for any \emph{square summable} input signal $u$, the norm of the output $Hu$ is at most $\gamma$ times the norm of the input, assuming that such $\gamma$ exists.

\begin{restatable}[\cite{hu2017control}, Theorem 2]{theorem}{maintheorem}
    \label{theorem:main_theorem}
    Consider an operator $F$ satisfying \Cref{assumption:sector} and \Cref{assumption:unique}, and let $\mathbf{K}'(z) = \mathbf{K}(z)(\mathbf{I} - h \mathbf{K}(z))^{-1}$, where $h = (L + \mu)/2$. If $||K'_{\rho} || = ||\mathbf{K}'(\rho z)|| < 2/(L - \mu)$, for some $\rho \in (0,1)$, and $\mathbf{K}(z)$ is strictly proper, then the optimization algorithm described with controller $K$ convergences linearly with rate $\rho$ to the unique fixed point of $F$.
\end{restatable}

We should note that $\mathbf{K}'(z)$ is called the \emph{complementary sensitivity matrix}, and it derives from a \emph{linear shift transformation} of the original feedback interconnection (see \Cref{appendix:lst}). Importantly, observe that this theorem essentially reduces establishing linear convergence to determining the \emph{gain} of the controller associated with the optimization method. To this end, we will use the following standard result:

\begin{restatable}{theorem}{Hinfinity}
    \label{theorem:H-infinity}
The $\ell_2$-gain of a stable (finite-order) LTI system $K$ equals its $\mathcal{H}_{\infty}$-norm, i.e.,

\begin{equation}
    ||K|| = \mathcal{H}_{\infty} = \sup_{\omega \in [-\pi, \pi]} \sigma_{max} \left( \mathbf{K}(e^{j \omega}) \right),
\end{equation}
where $\mathbf{K}(z)$ is the transfer matrix of the system, and $\sigma_{max}()$ is the maximum singular value.
\end{restatable}

In particular, the transfer matrix $\mathbf{K}(z)$ for all the optimization methods considered in this work can be expressed as $\mathbf{K}(z) = K(z) \mathbf{I}_d$; as such, applying \Cref{theorem:H-infinity} reduces to determining the maximum magnitude of $K(z)$ over the unit circle in the $z$-plane.

\paragraph{The Circle Criterion.} We will also verify some of our bounds using the so-called \emph{circle criterion} (sometimes referred to as the generalized Nyquist criterion). Specifically, the \emph{Nyquist plot} of a one-dimensional controller with transfer function $K(z)$ is the curve $\mathcal{C} := \{ (a, b) \in \mathbb{R}^2 : a + b j = K(e^{j \omega})\}$. In this context, if $D(a, b)$ represents the closed disk in the complex plane whose diameter coincides with the length of the line segment connecting the points $-1/a + j 0$ and $-1/b + j 0$, we will use the following result:

\begin{restatable}{proposition}{circlecriterion}
    \label{proposition:generalized_Nyquist}
        Consider an operator $F$ which satisfies \Cref{assumption:sector} and \Cref{assumption:unique}. Then, the system is absolutely stable if $K'(z)$ is stable and the Nyquist plot of $K'(e^{j \omega})$ lies in the interior of the disk $D(a, b)$, where $a = (L - \mu)/2 = -b$.
\end{restatable}

Recall that a rational function $K'(z)$ is said to be stable if the denominator constitutes a \emph{stable polynomial}, i.e. all of its roots lie within the unit circle in the $z$-plane. We give a comprehensive overview of the control-theoretic preliminaries required to establish these claims, as well as more rigorous definitions in \Cref{section:background}.

\section{Analysis of First-Order Methods}
\label{section:analysis}

In this section we employ the control-theoretic framework presented in \Cref{section:preliminaries} in order to analyze a series of optimization algorithms. As a warm-up, we commence with the analysis of Gradient Descent, which previously appeared in \cite{hu2017control}. A reader familiar with these techniques is welcome to skip the forthcoming subsection.

\subsection{Warm-up: Gradient Descent}

The purpose of this subsection is to illustrate in detail the basic steps required for the analysis of an optimization method using the framework introduced in the background section. Specifically, we will focus on perhaps the simplest optimization method, namely Gradient Descent (henceforth $\gd$), which boils down to the following update rule:

\begin{equation}
    \label{eq:GD}
    x_{k+1} = x_k - \eta  F(x_k),
\end{equation}
for some constant $\eta > 0$, typically referred to as the \emph{learning rate}. The first step of the analysis consists of casting the $\gd$ update rule in the form of a feedback interconnection, as follows: 

\begin{subequations}
\label{eq:dynamical_system-GD}
\begin{align}
    \xi_{k+1} &= \xi_k - \eta v_k \label{eq:GD-a} \\
    u_k &= \xi_k \label{eq:GD-b} \\
    v_k &= F(u_k + x^*). \label{eq:GD-c}
\end{align}
\end{subequations}

Specifically, note that \eqref{eq:GD-a} and \eqref{eq:GD-b} capture the LTI system $K$ (the controller), while \eqref{eq:GD-c} corresponds to the nonlinearity $P$. Indeed, this follows by directly substituting $\xi_k = x_k - x^*$. Now if we consider the input/output relation of the induced controller $K$, it follows that $u_k = \xi_k = \xi_{k-1} - \eta v_{k-1} = u_{k-1} - \eta v_{k-1}$. Thus, transforming to the $z$-space yields that $U(z) = z^{-1} U(z) - \eta z^{-1} V(z)$, or equivalently, 

\begin{equation}
    U(z) = - \frac{\eta}{z - 1} V(z).
\end{equation}
As a result, we have shown the following:

\begin{proposition}
    \label{proposition:GD-TF}
The transfer matrix of the Gradient Descent controller can be expressed as $\mathbf{K}(z) = K(z) \mathbf{I}_d $, where $K(z) = -\eta/(z-1)$.
\end{proposition}

This expression for the transfer function will be used directly to derive the gain of the controller via \Cref{theorem:H-infinity}. In particular, we can invoke \Cref{theorem:main_theorem} to derive linear convergence for the $\gd$ method along two well-known regimes:

\begin{theorem}[\cite{hu2017control}]
    Consider an operator $F$ which satisfies \Cref{assumption:sector} and \Cref{assumption:unique}. Then, if $\eta = 2/(\mu + L)$, Gradient Descent converges linearly to the fixed point of $F$ with rate $\rho$, for any $\rho > (L - \mu)/(L + \mu)$.
\end{theorem}

\begin{proof}
First, notice that the controller associated with $\gd$ is indeed strictly proper. Moreover, the transfer matrix $\mathbf{K}'(z)$ after applying the linear shift transformation can be expressed as 

\begin{equation*}
    \mathbf{K}'(z) =  \frac{K(z)}{1 - h K(z)} \mathbf{I}_d,
\end{equation*}
where $h = (L + \mu)/2$ and $K(z) = -\eta/(z-1)$. In particular, for $\eta = 2/(L + \mu) = h^{-1} $ it follows that 

\begin{equation*}
    K'(\rho z) = - \frac{\eta}{\rho z-1 + h \eta} = - \frac{\eta}{\rho z}.
\end{equation*}
Thus, the transfer function $K'(\rho z)$ is always stable, while \Cref{theorem:H-infinity} implies that $||K'_{\rho} || = \eta/\rho$. Consequently, \Cref{theorem:main_theorem} tells us that Gradient Descent exhibits linear convergence with rate $\rho$ as long as $\rho > (L - \mu)/(L + \mu)$, completing the proof.
\end{proof}

\begin{theorem}[\cite{hu2017control}]
    Consider an operator $F$ which satisfies \Cref{assumption:sector} and \Cref{assumption:unique}. Then, if $\eta = 1/L$, the Gradient Descent method converges linearly to the fixed point of $F$ with rate $\rho$, for any $\rho > 1 - \kappa^{-1}$.
\end{theorem}

\begin{proof}
Similarly to the previous proof, for $\eta := 1/L$ we obtain that 

\begin{equation*}
    K'(\rho z) = - \frac{\eta}{\rho z - \frac{1}{2}(1 - \kappa^{-1})}.
\end{equation*}

Therefore, it follows that $K'(\rho z)$ is stable for all $\rho > (1 - \kappa^{-1})/2$. Moreover, \Cref{theorem:H-infinity} implies that

\begin{equation*}
    || K'_{\rho} || = \frac{1}{L} \sup_{\omega \in [-\pi, \pi]} \frac{1}{ | \rho e^{j \omega} - \frac{1}{2}(1 - \kappa^{-1}) |} = \frac{1}{L} \frac{1}{\rho - \frac{1}{2}(1 - \kappa^{-1})}.
\end{equation*}
Thus, $\gd$ converges with linear rate $\rho$ as long as  

\begin{equation*}
    \frac{1 - \kappa^{-1}}{2 \rho - (1 - \kappa^{-1})} < 1 \iff \rho > 1 - \kappa^{-1}.
\end{equation*}

\end{proof}

\subsection{Optimistic Gradient Descent}

Next, we analyze the Optimistic Gradient Descent ($\ogd$) method, a variant of $\gd$ which augments the ``memory'' of the algorithm with an additional past-gradient. Specifically, the update rule of $\ogd$ boils down to the following equation:

\begin{equation}
    \label{eq:OGD}
    x_{k+1} = x_k - 2\eta F(x_{k}) + \eta F(x_{k-1}),
\end{equation}
while the associated controller can be expressed as follows:

\begin{subequations}
\label{eq:OGD-normal_form}
\begin{align}
    \begin{bmatrix}
    \xi_{k+1}^{(1)} \\
    \xi_{k+1}^{(2)}
    \end{bmatrix}
    &=
    \begin{bmatrix}
    \mathbf{I}_d & \eta \mathbf{I}_d \\
    \mathbf{0}_d & \mathbf{0}_d
    \end{bmatrix}
    \begin{bmatrix}
    \xi_{k}^{(1)} \\
    \xi_{k}^{(2)}
    \end{bmatrix}
    +
    \begin{bmatrix}
    -2\eta \mathbf{I}_d \\
    \mathbf{I}_d
    \end{bmatrix}
    v_k; \label{eq:OGD-a}
    \\
    u_k &=
    \begin{bmatrix}
    \mathbf{I}_d & \mathbf{0}_d
    \end{bmatrix}
    \begin{bmatrix}
    \xi_{k}^{(1)} \\
    \xi_{k}^{(2)}
    \end{bmatrix},
\end{align}
\end{subequations}
where $\xi_k^{(1)}, \xi_{k}^{(2)} \in \mathbb{R}^d$ represent the state variables of the controller. To verify this claim notice that from \Cref{eq:OGD-a} it follows that $\xi_{k+1}^{(1)} = \xi_k^{(1)} - 2\eta v_k + \eta v_{k-1}$, and it suffices to substitute $x_k = \xi_k^{(1)} + x^*$ to obtain the update rule of $\ogd$ \eqref{eq:OGD}; recall that $v_k = F(u_k + x^*)$. From this formulation the following observation is immediate:

\begin{proposition}
    \label{proposition:TF-OGD}
    The transfer matrix of the Optimistic Gradient Descent controller can be expressed as $\mathbf{K}(z) = K(z) \mathbf{I}_d$, where $K(z) = \eta (1-2z)/(z^2-z)$.
\end{proposition}
\begin{proof}
Let $(v, u)$ be the input/output pair of the $\ogd$ controller. \Cref{eq:OGD-normal_form} implies that $u_k = \xi_k^{(1)} = \xi_{k-1}^{(1)} + \eta v_{k-2} - 2\eta v_{k-1} \implies u_k = u_{k-1} - 2\eta v_{k-1} + \eta v_{k-2}$. Thus, taking the $z$-transform gives that $U(z) = z^{-1} U(z) - 2\eta z^{-1} V(z) + \eta z^{-2} V(z)$, or equivalently,

\begin{equation}
    U(z) = \eta \frac{1 - 2z}{z^2 - z} V(z).
\end{equation}
\end{proof}

Next, having derived a succinct representation for the controller of $\ogd$, we are ready to characterize the linear convergence of $\ogd$:

\begin{theorem}
    \label{theorem:OGD}
    Consider an operator $F$ which satisfies \Cref{assumption:sector} and \Cref{assumption:unique}. Then, if $\eta = \frac{2}{3L}(1-\epsilon)$, for some $\epsilon \in (0,1)$, Optimistic Gradient Descent \eqref{eq:OGD} converges linearly to the fixed point of $F$ with rate $\rho = 1 - \frac{2}{3} \epsilon(1-\epsilon)\kappa^{-1}$.
\end{theorem}

\begin{proof}
First, we investigate the stability of the complementary sensitivity function of the OGD controller. For convenience, let us denote with $\lambda = \kappa^{-1} \in (0,1)$. For $\eta = \frac{2}{3L}(1-\epsilon)$ we obtain that 

\begin{equation*}
    K'(\rho z) = \eta \frac{1 - 2\rho z}{\rho^2 z^2 + \frac{2\lambda - 2\epsilon \lambda - 2 \epsilon - 1}{3} \rho z - \frac{(1-\epsilon)(1 + \lambda)}{3}}.
\end{equation*}
Let $z_0, z_1$ be the roots of the characteristic equation 

\begin{equation}
    \label{eq:quadr}
 z^2 + \frac{2\lambda - 2\epsilon \lambda - 2 \epsilon - 1}{3} z - \frac{(1-\epsilon)(1 + \lambda)}{3} = 0.   
\end{equation}

Then, it follows that $K'(\rho z)$ is stable if $\rho > \max \{|z_0|, |z_1| \}$. Moreover, solving the quadratic equation \eqref{eq:quadr} yields 

\begin{equation*}
    z_{0, 1} = \frac{1}{6} \left( - 2\lambda + 2\epsilon \lambda + 2\epsilon + 1 \pm \sqrt{(-2\lambda + 2\epsilon \lambda + 2\epsilon + 1)^2 + 12(1-\epsilon)(1 + \lambda)} \right).
\end{equation*}

Next, we will use the following simple lemma in order to bound the magnitudes of the solutions to the characteristic equation:

\begin{lemma}
    \label{lemma:aux-OGD}
    For any $\epsilon \in (0,1)$ and $\lambda \in (0,1)$, the following inequalities hold:
    
    \begin{equation*}
        \sqrt{(-2\lambda + 2\epsilon \lambda + 2\epsilon + 1)^2 + 12(1-\epsilon)(1 + \lambda)} < 2 \lambda - 2\epsilon \lambda - 2 \epsilon + 5 - 4 \epsilon (1 - \epsilon) \lambda;
    \end{equation*}
    
    \begin{equation*}
        \sqrt{(-2\lambda + 2\epsilon \lambda + 2\epsilon + 1)^2 + 12(1-\epsilon)(1 + \lambda)} < - 2 \lambda + 2\epsilon \lambda + 2 \epsilon + 7 - 4 \epsilon (1 - \epsilon) \lambda.
    \end{equation*}
\end{lemma}

In particular, it is easy to see that this lemma implies that $|z_0| < 1 - \frac{2}{3} \epsilon (1 - \epsilon) \lambda$ and $|z_1| < 1 - \frac{2}{3} \epsilon (1 - \epsilon) \lambda$. Therefore, it suffices to take $\rho \geq 1 - \frac{2}{3} \epsilon (1 - \epsilon) \lambda$ to ensure that $K'(\rho z)$ is stable. The next step is to bound the gain of $K'_{\rho}$. To this end, \Cref{theorem:H-infinity} implies that 

\begin{equation*}
    ||K'_{\rho}|| = \eta  \sup_{\omega \in [-\pi, \pi]} \left| \frac{ 1 - 2 \rho e^{j \omega }}{\rho^2 e^{2j \omega} + \frac{2 \lambda - 2 \epsilon \lambda - 2 \epsilon - 1}{3} \rho e^{j \omega} - \frac{(1 - \epsilon)(1 + \lambda)}{3}} \right|.
\end{equation*}
Moreover, simple calculations yield that 

\begin{equation}
    \label{eq:psi}
    \left| \frac{ 1 - 2 \rho e^{j \omega }}{\rho^2 e^{2j \omega} + \frac{2 \lambda - 2 \epsilon \lambda - 2 \epsilon - 1}{3} \rho e^{j \omega} - \frac{(1 - \epsilon)(1 + \lambda)}{3}} \right|^2 = \frac{a_0 + a_1 \cos(\omega)}{b_0 + b_1 \cos(\omega) + b_2 \cos(2\omega)} \define \psi(\omega),
\end{equation}
where
\begin{subequations}
\begin{align*}
    a_0 &= 1 + 4\rho^2; \\
    a_1 &= - 4 \rho; \\
    b_0 &= \rho^4 + \frac{(1-\epsilon)^2 (1 + \lambda)^2}{9} + \frac{(2\lambda - 2 \epsilon \lambda - 2 \epsilon - 1)^2}{9}; \\
    b_1 &= \frac{2}{3} (2\lambda - 2 \epsilon \lambda - 2 \epsilon - 1) \left( \rho^2 - \frac{(1 - \epsilon)(1 + \lambda)}{3} \right); \\
    b_2 &= - \frac{2}{3} \rho^2 (1 - \epsilon) (1 + \lambda).
\end{align*}
\end{subequations}

\begin{lemma}
    For $\rho = 1 - \frac{2}{3} \epsilon (1 - \epsilon) \lambda$, the function $\psi(\omega)$ as defined in \eqref{eq:psi} attains its maximum either on $\omega = 0$ or $\omega = \pi$.
\end{lemma}

This claim can be verified using elementary calculus. As a result, this lemma implies that for $\rho = 1 - \frac{2}{3} \epsilon(1 - \epsilon) \lambda$,

\begin{equation}
    \label{eq:K-max}
    || K'_{\rho} || = \max \{  | K'(\rho) |, | K'(-\rho) | \}.
\end{equation}

Finally, to verify that the condition of \Cref{theorem:main_theorem} regarding the gain of the operators is met, we will use the following lemma: 

\begin{lemma}
    Let $\lambda \in (0,1)$ and $\epsilon \in (0,1)$. For $\rho = 1 - \frac{2}{3} \epsilon(1-\epsilon) \lambda$ the following inequalities hold:
    
    \begin{equation}
        \label{eq:OGD-b_1}
        |(1 - 2\rho)|(1-\epsilon)(1 - \lambda) < 3 \left| \rho^2 + \rho \frac{2 \lambda - 2 \epsilon \lambda - 2 \epsilon - 1}{3} - \frac{(1 - \epsilon)(1 + \lambda)}{3} \right|; 
    \end{equation}
    
    \begin{equation}
        \label{eq:OGD-b_2}
        |(1 + 2\rho)|(1-\epsilon)(1 - \lambda) < 3 \left| \rho^2 - \rho \frac{2 \lambda - 2 \epsilon \lambda - 2 \epsilon - 1}{3} - \frac{(1 - \epsilon)(1 + \lambda)}{3} \right|. 
    \end{equation}
\end{lemma}

Therefore, \eqref{eq:OGD-b_1} implies that $ |K'(\rho)| (L - \mu)/2 < 1 $, while \eqref{eq:OGD-b_2} gives that $|K'(-\rho)| (L - \mu)/2 < 1$. As a result, \eqref{eq:K-max} implies that $|| K'_{\rho} || < 2/(L - \mu)$, for $\rho = 1 - \frac{2}{3} \epsilon (1 - \epsilon) \lambda $, and the theorem follows from \Cref{theorem:main_theorem}.
\end{proof}

Importantly, the bound we obtained with respect to the learning rate is tight. That is, there exist operators satisfying \Cref{assumption:sector} and \Cref{assumption:unique} such that for $\eta > 2/(3L)$ the $\ogd$ dynamics are unstable. Indeed, this is shown in the following proposition:

\begin{restatable}{proposition}{tightness}
    \label{proposition:OGD-tightness}
    Consider the function $f: \mathbb{R} \ni x \mapsto x^2 - \cos(x)$. For any $\eta > 2/9$, Optimistic Gradient Descent diverges under any initial state such that $(x_0, x_{-1}) \neq (0,0)$. 
\end{restatable}

The proof of this proposition is deferred to \Cref{appendix:proof-tightness}. Observe that $f$ has a unique minimum at $0$, while $\nabla f$ has a Lipschitz constant of $3$. Thus, \Cref{theorem:OGD} predicts that under $\eta < 2/9 = 2/(3L)$ the $\ogd$ dynamics are stable; along with \Cref{proposition:OGD-tightness}, this implies that our bound with respect to the learning rate in \Cref{theorem:OGD} is tight. Naturally, this example can be directly extended to the min-max setting by considering the induced \emph{uncoupled} objective function.

Before we conclude this subsection we apply the circle criterion to verify that $\ogd$ converges to the fixed point when $\eta < 2/(3L)$; see \Cref{fig:Nyquist-OGD} for a graphical illustration.

\begin{figure}[!ht]
    \centering
    \includegraphics[scale=0.4]{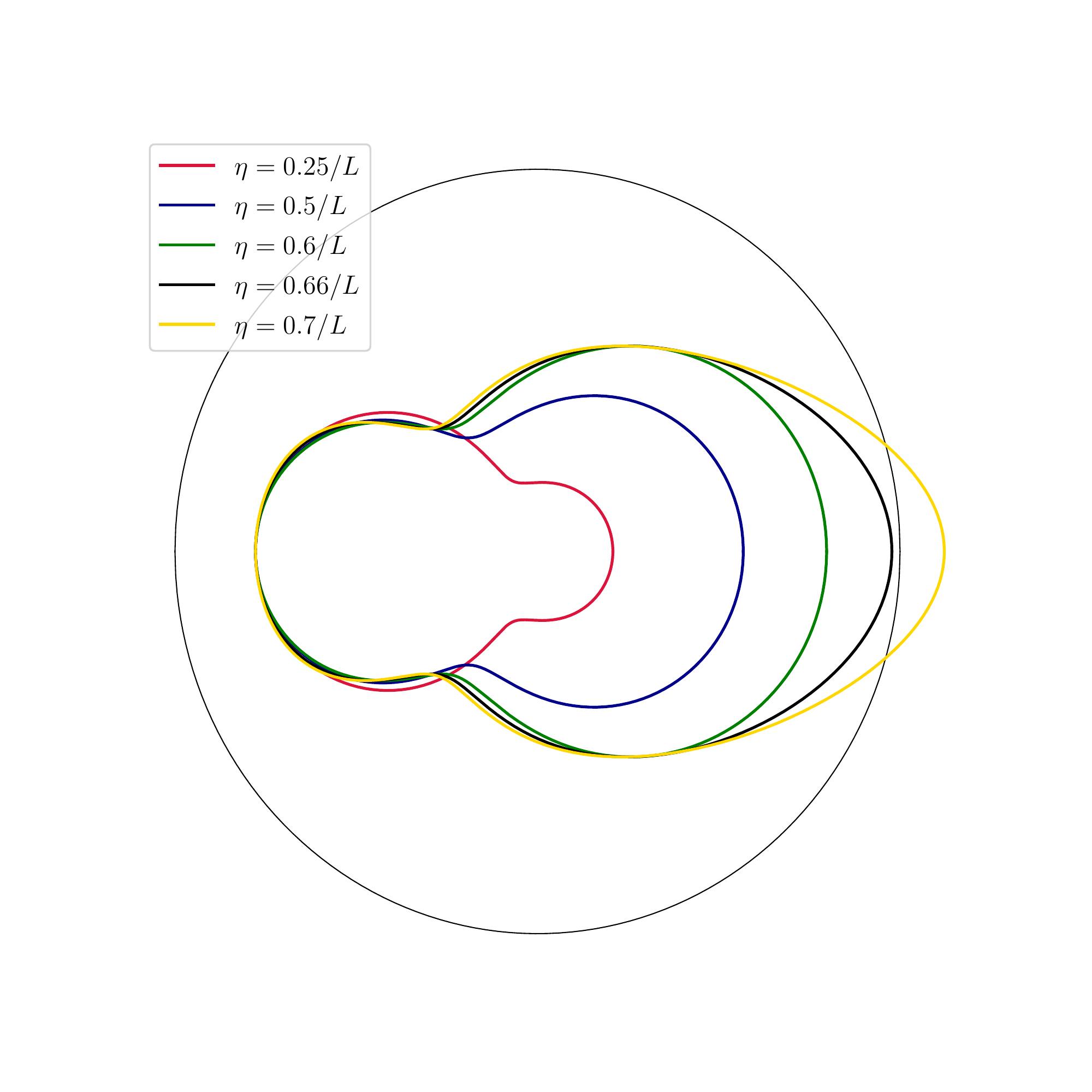}
    \caption{A graphical illustration of the circle criterion for Optimistic Gradient Desent under different values of learning rate: $\eta \in \{0.25/L, 0.5/L, 0.6/L, 0.66/L, 0.7/L\}$. The $x$-axis and $y$-axis represent $\Re\{ K' (e^{j \omega}) \}$ and $\Im \{ K'(e^{j \omega}) \}$ respectively. The circle corresponds to the closed disk $D(a,b)$ for $a = -(L - \mu)/2 = -b$. We can indeed verify that the Nyquist plot of $K'$ lies in the interior of $D(a,b)$ for $\eta <2/(3L)$. We note that these experiments were conducted with $\mu = 0.5$ and $L = 4$.}
    \label{fig:Nyquist-OGD}
\end{figure}

\begin{remark}[Other Single-Call Variants of Extra-Gradient]
\label{remark:extension}
Naturally, the results we establish for $\ogd$ apply for other equivalent and well-studied single-call variants of the Extra-Gradient method, namely Past Extra-Gradient Descent and Reflected Gradient Descent. This equivalence is shown in \Cref{appendix:single_call} using the proposed control-theoretic approach.
\end{remark}

\subsection{Generalized Optimistic Gradient Descent}

We also generalize our analysis for the Optimistic Gradient Descent method, considering the following extended dynamics:

\begin{equation}
    \label{eq:GOGD}
    x_{k+1} = x_k - (\alpha + \beta) F(x_k) + \beta F(x_{k-1}),
\end{equation}
where $\alpha > 0, \beta \geq 0$ are the parameters of the Generalized Optimistic Gradient Descent method (henceforth $\gogd$). Notice that for $\beta = 0$ we recover GD with $\eta = \alpha$, while for $\alpha = \beta$ we recover $\ogd$ with $\eta = \alpha = \beta$. Thus, $\gogd$ interpolates $\gd$ with $\ogd$. Analogously to our analysis for OGD, we can show the following:

\begin{proposition}
    \label{proposition:TF-GOGD}
    The transfer matrix of the Generalized Optimistic Gradient Descent controller can be expressed as $\mathbf{K}(z) = K(z) \mathbf{I}_d$, where $K(z) = (\beta - (\alpha + \beta)z)/(z^2-z)$.
\end{proposition}

The proof of this proposition follows identically to \Cref{proposition:TF-OGD}. Thus, we are ready to characterize the behavior of the $\gogd$ method:

\begin{restatable}{theorem}{genogd}
    \label{theorem:GOGD}
    Consider an operator $F$ which satisfies \Cref{assumption:sector} and \Cref{assumption:unique}. Then, if $\alpha = 1/(2L)$ and $\beta = \ell/(2L)$ for some $\ell \in [0,1]$, the Generalized Optimistic Gradient Descent method \eqref{eq:GOGD} converges linearly to the fixed point of $F$ with rate $\rho = 1 - \kappa^{-1}/4$.
\end{restatable}

The proof is included in \Cref{appendix:proof-GOGD}, as it follows analogously to \Cref{theorem:OGD}. We also provide a characterization under a different regime of parameters:
\begin{theorem}
    \label{theorem:GOGD-2}
    Consider an operator $F$ which satisfies \Cref{assumption:sector} and \Cref{assumption:unique}. Then, if $\alpha = 1/L$ and $\beta = \epsilon/(2L)$ for some $\epsilon \in (0,1)$, the Generalized Optimistic Gradient Descent method \eqref{eq:GOGD} converges linearly to the fixed point of $F$ with rate $\rho = 1 - \epsilon (1-\epsilon) \kappa^{-1}/2$.
\end{theorem}

The proof of this theorem proceeds similarly to \Cref{theorem:OGD} and \Cref{theorem:GOGD}, and it is therefore omitted. Before we conclude this section we illustrate the behavior of the $\gogd$ controller using the circle criterion (\Cref{proposition:generalized_Nyquist}). In particular, \Cref{fig:Nyquist-GOGD} illustrates the transition from $\gd$ to $\ogd$.

\begin{figure}
\centering
\begin{minipage}{.5\textwidth}
  \centering
  \includegraphics[scale=0.33]{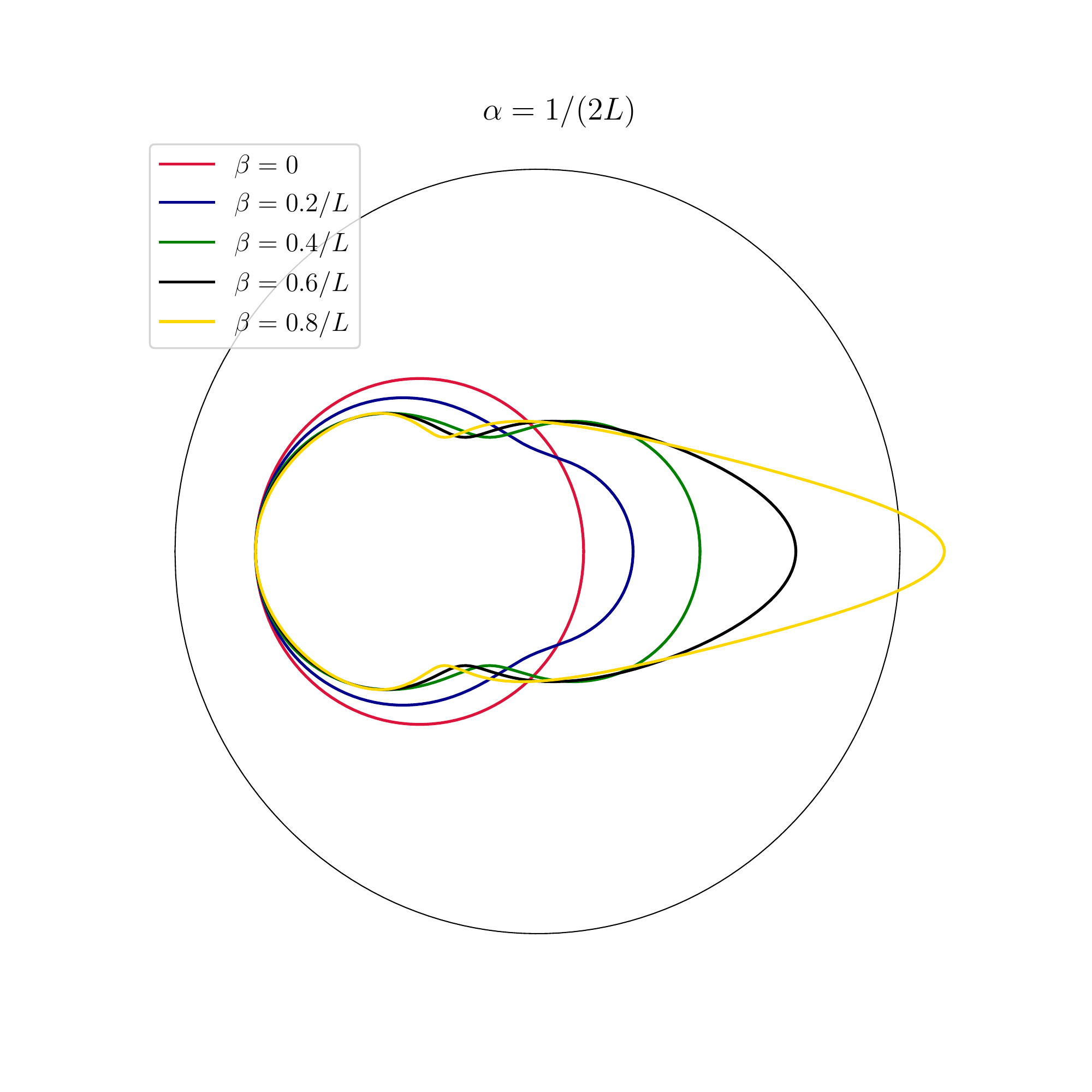}
  \label{fig:test1}
\end{minipage}%
\begin{minipage}{.5\textwidth}
  \centering
  \includegraphics[scale=0.33]{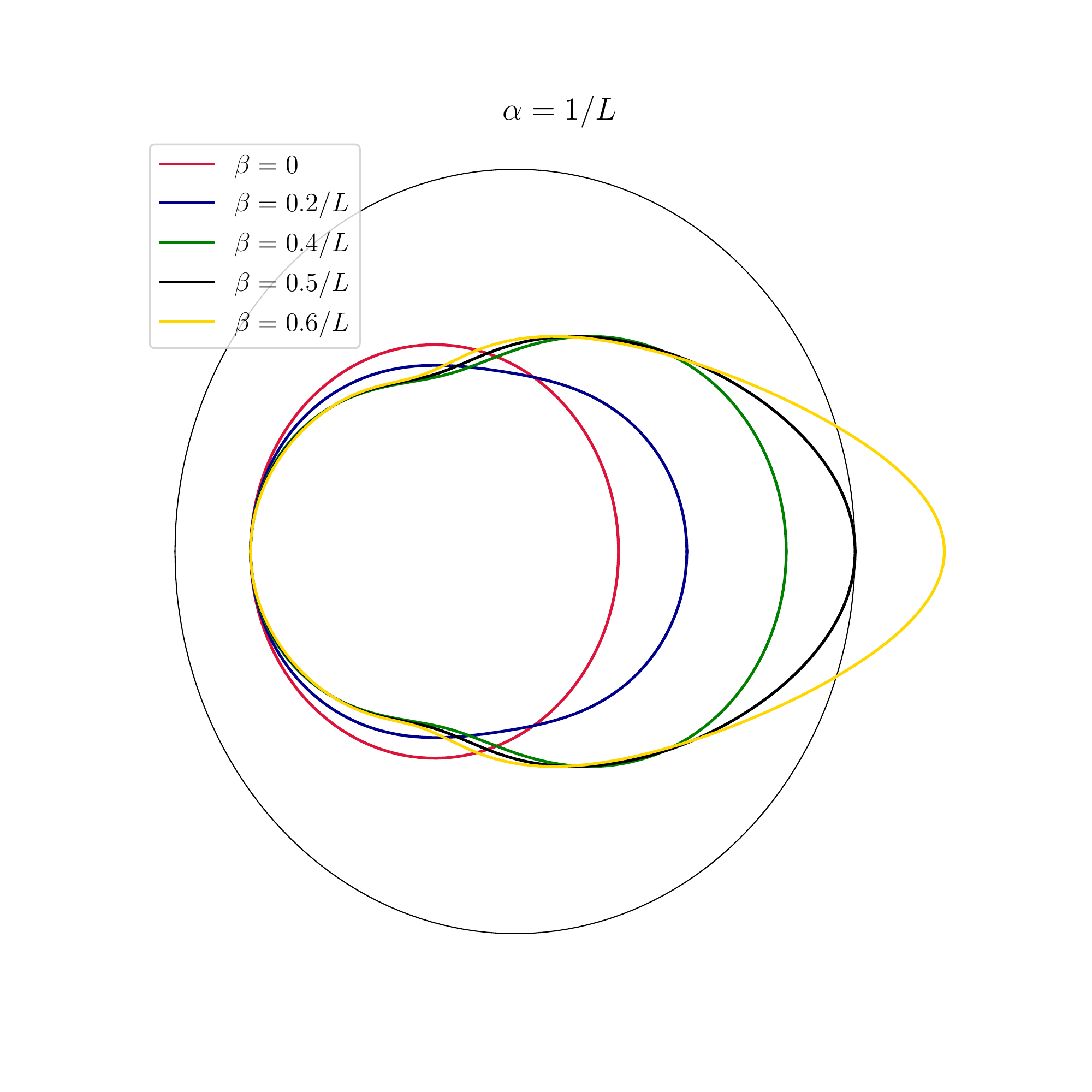}
  \label{fig:test2}
\end{minipage}
\caption{A graphical illustration of the circle criterion for the Generalized Optimistic Gradient Descent ($\gogd$) controller for different values of $\alpha$ and $\beta$.}
\label{fig:Nyquist-GOGD}
\end{figure}

\subsection{Proximal Point Method}

The Proximal Point method (PP) is one of the most well-studied \emph{implicit} algorithms in optimizing, commemcing with the seminal work of Rockafellar \cite{doi:10.1137/0314056,10.2307/3689277} (see also \cite{doi:10.1137/1.9781611974997}). For the sake of simplicity, in this section we assume that the operator $F$ corresponds to the gradient of a continuously differentiable, smooth and convex function $f$. In this context, the update of $\pp$ boils down to the following equation:

\begin{equation}
    \label{eq:PP}
    x_{k+1} = \prox_{\eta f} (x_k) = \argmin_{x \in \mathbb{R}^d} \left\{ f(x) + \frac{1}{2\eta} || x - x_k ||^2 \right\}.
\end{equation}

It is easy to see that \eqref{eq:PP} is tantamount to $x_{k+1} = x_k - \eta \nabla f(x_{k+1})$. Thus, the $\pp$ controller can be expressed in the following normal form: 

\begin{subequations}
\label{eq:PP-normal_form}
\begin{align}
    \xi_{k+1} = \xi_k - \eta v_k; \\
    u_k = \xi_k - \eta v_k,
\end{align}
\end{subequations}
where $v_k = \nabla g(u_k + x^*)$, and $\xi_k \in \mathbb{R}^d$ represents the state vector of the controller. Notice that unlike the previously analyzed optimization algorithms, $\mathbf{D}$ is \emph{not} the zero matrix, and subsequently the transfer matrix is \emph{not} strictly proper. Nonetheless, we claim that when $f$ is convex, the induced feedback interconnection is indeed well-posed, and the same holds after we apply the linear shift transformation.

\begin{proposition}
    \label{proposition:PP-TF}
The transfer function of the Proximal Point controller can be expressed as $\mathbf{K}(z) = K(z) \mathbf{I}_d$, where $K(z) = -\eta z/(z-1)$.
\end{proposition}

\begin{proof}
It follows that $u_k = u_{k-1} - \eta v_k$, and transferring to the $z$-space leads to the desired conclusion.
\end{proof}

\begin{theorem}
    Consider a continuously differentiable function $f$ whose gradients satisfy \Cref{assumption:sector} and \Cref{assumption:unique}. Then, if $\eta = 2t/(L + \mu)$ for some $t > 0$, the Proximal Point method converges linearly to $x^*$ with rate $\rho$, for any $\rho > (L + \mu)/(L + \mu + 2\mu t)$.
\end{theorem}

\begin{proof}
The complementary sensitivity function of the $\pp$ controller reads 

\begin{equation*}
    K'(\rho z) = - \eta \frac{\rho z}{\rho z - 1 + h \eta \rho z} = -\eta \frac{\rho z}{\rho z (1 + t) - 1}.
\end{equation*}

Thus, for any $\rho > (L + \mu)/(L + \mu + 2\mu t) \geq 1/(1 + t)$ it follows that $K'(\rho z)$ is stable. Moreover, it is easy to see that $|| K'_{\rho} || = \eta \rho / (\rho (1+t) - 1)$, and simple calculations imply that 

\begin{equation*}
    || K_{\rho}' || \left( \frac{L - \mu}{2} \right) = \frac{t \rho}{t \rho + \rho -1} \frac{L - \mu}{L + \mu} < 1;
\end{equation*}
thus, the conditions of \Cref{theorem:main_theorem} are satisfied, and the theorem follows.
\end{proof}

Therefore, if the learning rate $\eta$ is sufficiently large, the Proximal Point method converges linearly with any arbitrary rate, while for $\eta \to \infty$ the convergence is \emph{superlinear}; these observations are consistent with the seminal work of Rockafellar \cite{doi:10.1137/0314056}. For completeness, we also provide the following theorem without a proof:

\begin{theorem}
    Consider a continuously differentiable function $f$ whose gradients satisfy \Cref{assumption:sector} and \Cref{assumption:unique}. Then, if $\eta = t/L$ for some $t > 0$, the Proximal Point method converges linearly to $x^*$ with rate $\rho$, for any $\rho > 1/(1 + \kappa^{-1} t) $.
\end{theorem}

\subsection{PID Control}

In this subsection we show that all of the previously analyzed first-order methods are actually instances of $\pid$ control. In particular, first note that in our setting the desired \emph{reference point} corresponds to $F = 0$, and hence, the \emph{error signal} can be expressed as $e_k = 0 - v_k$, where $v_k$ is the output of the nonlinearity (the plant). In this context, the (discrete) $\pid$ controller can be expressed with the following input/output relation:

\begin{equation}
    \label{eq:PID-control}
    u_k = \overbrace{K_P e_k}^{\textrm{Proportional}} + \overbrace{K_I \sum_{i=1}^{\infty} e_{k-i}}^{\textrm{Integral}} + \overbrace{K_D (e_k - e_{k-1})}^{\textrm{Derivative}},
\end{equation}
where $(K_P, K_I, K_D)$ represent the parameters of the controller. Here we should stress that different variants of $\pid$ control arise depending on the implementation of the discrete-time derivative and integration, but the form of \Cref{eq:PID-control} is perhaps the most standard one. Next, if we transform \eqref{eq:PID-control} to the $z$-space, and we use that $e_k = - v_k$, we reach to the following conclusion:

\begin{proposition}
    \label{proposition:PID-TF}
The transfer matrix of the $\pid$ controller can be expressed as $\mathbf{K}(z) = K(z) \mathbf{I}_d$, where

\begin{equation}
    K(z) = - \frac{(K_P + K_D)z^2 + (-K_P + K_I - 2 K_D)z + K_D}{z^2 - z}.
\end{equation}
\end{proposition}

As a result, we will use this proposition to conclude that $\pid$ control subsumes all of the previously analyzed first-order methods:

\begin{proposition}
    \label{proposition:PID-GOGD}
    $\pid$ control with $(K_P, K_I, K_D) = (\beta, \alpha, - \beta)$ is equivalent to the Generalized Optimistic Gradient Descent method \eqref{eq:GOGD} with parameters $(\alpha, \beta)$.
\end{proposition}

\begin{proposition}
    \label{proposition:PID-PP}
    $\pid$ control with $(K_P, K_I, K_D) = (\eta, \eta, 0)$ is equivalent to the Proximal Point method \eqref{eq:PP} with learning rate $\eta$.
\end{proposition}

\begin{remark}[Continuous-Time Dynamics]
All of the results presented regarding discrete-time dynamics are directly applicable for continuous-time dynamics as well. More precisely, the most common discretization in control theory is the bilinear transform, wherein parameter $z$ is associated with the continuous-time frequency $s$ of the Laplace transform, such that $z = (1 - sT/2)/(1 + sT/2)$, where $T$ represents the sampling period. Under this M\"{o}bius transform, the unit-circle in the $z$-plane, $|z| = 1$, is mapped to the imaginary axis in the $s$-plane, $\Re[s] = 0$, and it is well-known that the stability, as well as the gain of the controller are retained. 
\end{remark}

\subsection{Historical Methods}

In this subsection we explain how the approach previously employed can be numerically automated for a generic class of first-order algorithms. Specifically, we consider the following class of optimization methods:

\begin{equation}
    \label{eq:HGD}
    x_{k + T} = x_{k + T - 1} - \eta \sum_{i=1}^T a_i F(x_{k+T - i}),
\end{equation}
where $\eta > 0$ is the learning rate, and $a \in \mathbb{R}^T$ represents a given vector of parameters under some \emph{time horizon} $T \in \mathbb{N}$; you may assume that $||a|| = 1$. This method will be referred to as \emph{historical gradient descent} ($\hgd$). Naturally, this class includes as special cases algorithms such as the optimistic method ($\ogd$).

\begin{proposition}
    \label{proposition:HGD-TF}
The transfer matrix of the $\hgd$ controller can be expressed as $\mathbf{K}(z) = K(z) \mathbf{I}_d$, where

\begin{equation}
    K(z) = - \eta \frac{a_T + a_{T-1} z + \dots + a_1 z^{T-1}}{z^T - z^{T-1}}.
\end{equation}
\end{proposition}

Now assume that we are given some learning rate $\eta = \eta_0/L$ and a parameter $\rho \in (0,1)$; it will also be assumed that the the parameters of the nonlinearity---a lower bound on $\mu$ and an upper bound on $L$---are known. We will explain how to automatically evaluate whether the conditions of \Cref{theorem:main_theorem} are satisfied. In particular, this process essentially consists of two steps:

\paragraph{Stability of a Polynomial.} First, after we apply the linear shift transformation, the characteristic equation reads

\begin{equation*}
    (\rho z)^T - (\rho z)^{T-1} + \frac{\eta_0}{2} (1 + \kappa^{-1}) (a_T + a_{T-1} (\rho z) + \dots + a_1 (\rho z)^{T-1}) = 0,
\end{equation*}
where $\kappa = L/\mu$. In this context, determining whether the roots of a polynomial with given coefficients lie inside the unit circle in the $z$-plane---i.e. testing the stability of a polynomial---constitutes one of the most well-studied problems in control theory, and can be solved efficiently, for example, via Bistritz's method \cite{1457261,989164} (under discrete-time dynamics), or many other schemes. 

\paragraph{Gain of the Controller.} Assuming that the controller is stable, the next step is to determine its gain, which boils down to maximizing the following function:

$$
    | K'(e^{j \omega}) |^2 = \left\lvert \frac{ a_T + a_{T-1} (\rho e^{j \omega}) + \dots + a_1 (\rho e^{j \omega})^{T-1}}{(\rho e^{j \omega})^T - (\rho e^{j \omega})^{T-1} + \frac{\eta_0}{2} (1 + \kappa^{-1}) (a_T + a_{T-1} (\rho e^{j \omega}) + \dots + a_1 (\rho e^{j \omega})^{T-1})} \right\rvert^2,
$$
for $\omega \in [0, 2\pi]$. In particular, it is easy to see that if we employ the formula $cos ((n_1 - n_2) \omega) = \cos(n_1 \omega) \cos(n_2 \omega) + \sin(n_1 \omega) \sin(n_2 \omega)$, and we use the fact that $\cos(n\omega)$ can be expressed as an $n$-th degree polynomial of $\cos(x)$, it suffices to solve the following optimization problem:

\begin{equation}
    \max_{x \in [-1, 1]} \frac{P(x; \rho, \eta_0, \kappa, T)}{Q(x; \rho, \eta_0, \kappa, T)},
\end{equation}
where $P(x)$ and $Q(x)$ are (univariate) polynomials; this new formulation follows directly from the substitution $x := \cos(\omega)$. Observe that this problem is indeed well-posed as we have ensured---by virtue of stability---that $Q$ does not have any real roots on $[-1, 1]$. Hence, one can efficiently \emph{approximate} the maximum of this rational function using an array of standard methods. Consequently, for a given $\eta_0$ and $\rho$ we can evaluate whether the conditions of \Cref{theorem:main_theorem} are met, and then it suffices to perform a bisection search on the parameters $\eta_0$ and $\rho$ in order to identify a region of stability, as well as a rate of convergence for the optimization method---assuming that such a region indeed exists. 

\begin{remark}
For the sake of simplicity we have focused on algorithms that take the form of \Cref{eq:HGD} (i.e. $\hgd$), but in fact, our reduction can be applied more broadly for the following class of ``historical'' gradient-based methods:

\begin{equation}
    x_{k + T} = \sum_{i=1}^T b_i x_{k + T - i} - \eta \sum_{i=1}^T a_i F(x_{k+T - i}).
\end{equation}
\end{remark}

\subsection{Alternating Optimistic Gradient Descent}

In the context of min-max optimization, the analysis we presented for $\ogd$ applies for the \emph{simultaneous} dynamics, i.e. both players update their strategies at the same time based on the information they have about the current state. However, there are many scenarios in which the dynamics proceed in an \emph{alternating} fashion: The second player takes into account the update of the opponent before performing the optimization step. The purpose of this subsection is to employ a frequency-domain representation in order to quantitatively compare these dynamics. In particular, we will focus solely on the usual \emph{bilinear} setting, i.e. $f(x,y) = x^{\top} \mathbf{A} y$. For simplicity, it will be tacitly assumed that the real matrix $\mathbf{A}$ is $n \times n$ (square) and non-singular; this hypothesis comes without loss of generality since we can always reduce to this case via a linear transformation of the original system. The simultaneous and the alternating dynamics (respectively) take the following form:

\begin{align}
\label{eq:sim-ogd}
\tag{Sim-OGD}
\begin{split}
 x_{k+1} &= x_k - 2\eta \mathbf{A} y_{k} + \eta \mathbf{A} y_{k-1}; \\
 y_{k+1} &= y_k + 2\eta \mathbf{A}^{\top} x_k - \eta \mathbf{A}^{\top} x_{k-1}.
\end{split}
\end{align}

\begin{align}
\label{eq:alt-ogd}
\tag{Alt-OGD}
\begin{split}
 x_{k+1} &= x_k - 2\eta \mathbf{A} y_{k} + \eta \mathbf{A} y_{k-1}; \\
 y_{k+1} &= y_k + 2\eta \mathbf{A}^{\top} x_{k+1} - \eta \mathbf{A}^{\top} x_{k}.
\end{split}
\end{align}

Let us focus on the alternating dynamics. In particular, if we transfer \eqref{eq:alt-ogd} in the $z$-domain, we arrive at the following conclusion (see \Cref{appendix:proof-alt-ogd}):

\begin{restatable}{proposition}{altch}
    \label{proposition:alt-ogd}
    Let $\alpha(z)$ be the characteristic polynomial of matrix $\mathbf{A} \mathbf{A}^{\top}$. Then, the characteristic equation of \eqref{eq:alt-ogd} can be expressed as 
    
    \begin{equation}
        \label{eq:alt-ogd-ch}
        \det( (z-1)^2 \mathbf{I}_n - \eta^2 \mathbf{A} \mathbf{A}^{\top} (2z-1)(-2+z^{-1})) = 0 \iff \alpha \left(-z \left(\frac{1}{\eta} \frac{z-1}{2z-1} \right)^2 \right) = 0.
    \end{equation}
\end{restatable}

For comparison, the characteristic equation of the simultaneous dynamics takes the following form:

\begin{proposition}[\cite{anagnostides2020robust}]
    \label{proposition:sim-ogd}
    Let $\alpha(z)$ be the characteristic polynomial of matrix $\mathbf{A} \mathbf{A}^{\top}$. Then, the characteristic equation of \eqref{eq:sim-ogd} can be expressed as 
    
    \begin{equation}
        \label{eq:sim-ogd-ch}
        \det( (z-1)^2 \mathbf{I}_n - \eta^2 \mathbf{A} \mathbf{A}^{\top} (2-z^{-1})(-2+z^{-1})) = 0 \iff \alpha \left(- \left(\frac{1}{\eta} \frac{z^2-z}{2z-1} \right)^2 \right) = 0.
    \end{equation}
\end{proposition}

This characteristic equation should look very reminiscent to the transfer function we derived for $\ogd$ (\Cref{tab:transfer_function}). The behavior of \Cref{eq:alt-ogd-ch} and \Cref{eq:sim-ogd-ch} is depicted in \Cref{fig:spectrum_transformation}. In particular, it follows that \eqref{eq:alt-ogd} is stable if for all the eigenvalues $\Lambda$ of $\mathbf{A} \mathbf{A}^{\top}$, $\eta \sqrt{\lambda} < 2/3, \forall \lambda \in \Lambda$. In turn, this implies the following theorem:

\begin{theorem}
    The alternating Optimistic Gradient Descent dynamics \eqref{eq:alt-ogd} converge linearly to the unique Nash equilibrium $(0, 0)$ if $\eta < 2/(3\gamma)$, where $\gamma$ is the spectral norm of $\mathbf{A}$.
\end{theorem}

\begin{figure}
    \centering
    \includegraphics[scale=0.4]{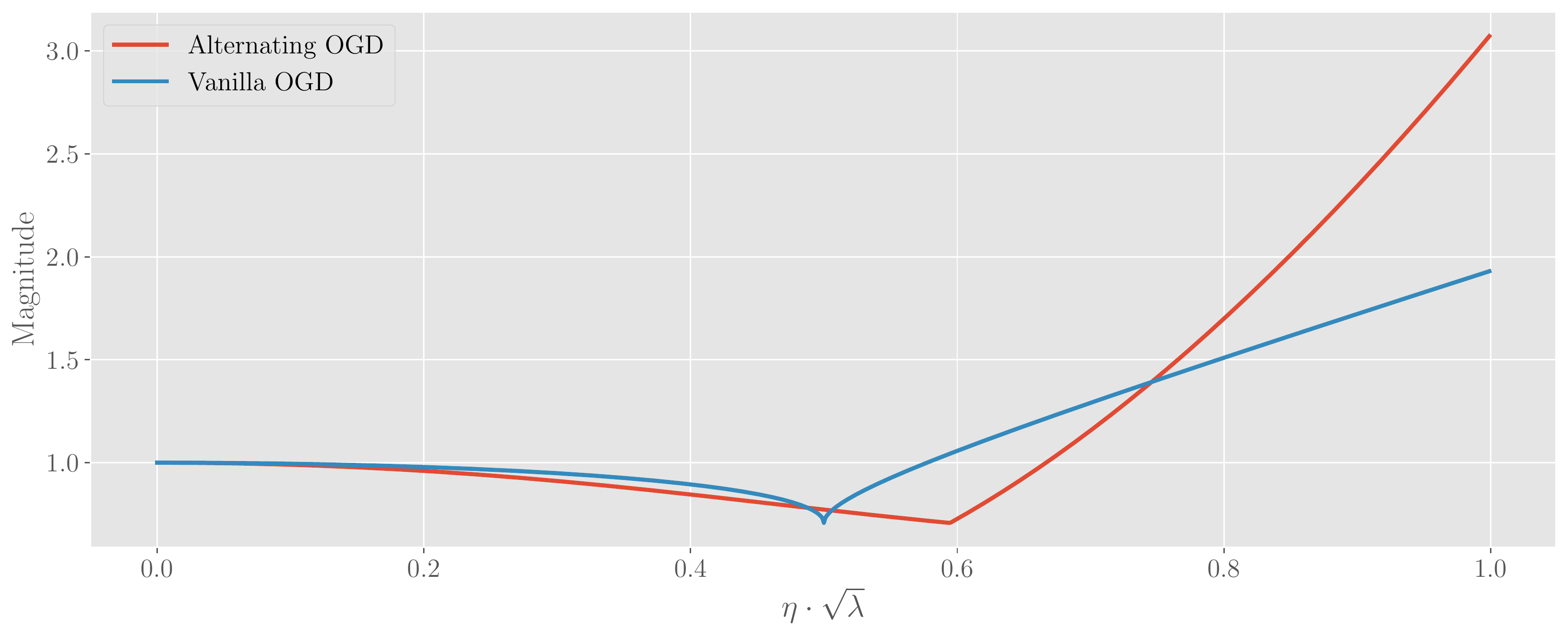}
    \caption{The spectrum transformation of \eqref{eq:alt-ogd} and \eqref{eq:sim-ogd}. More precisely, for a fixed learning rate $\eta$ and a value $\lambda > 0$ such that $\alpha(\lambda) = 0$, this figures illustrates the induced solution in the characteristic equation, as predicted by \Cref{proposition:alt-ogd} and \Cref{proposition:sim-ogd}. Notice that the $x$-axis represents the value of $\eta \sqrt{\lambda}$, while the $y$-axis indicates the maximum magnitude of the induced solutions.}
    \label{fig:spectrum_transformation}
\end{figure}

For the simultaneous dynamics the region of stability corresponds to $\eta < 1/(\sqrt{3} \gamma)$ \cite{anagnostides2020robust}. Interestingly, \Cref{fig:spectrum_transformation} indicates that the alternating dynamics typically exhibit faster convergence, as corroborated by some empirical data (\Cref{fig:exp-alt}).

\begin{figure}[!ht]
    \centering
    \includegraphics[scale=0.4]{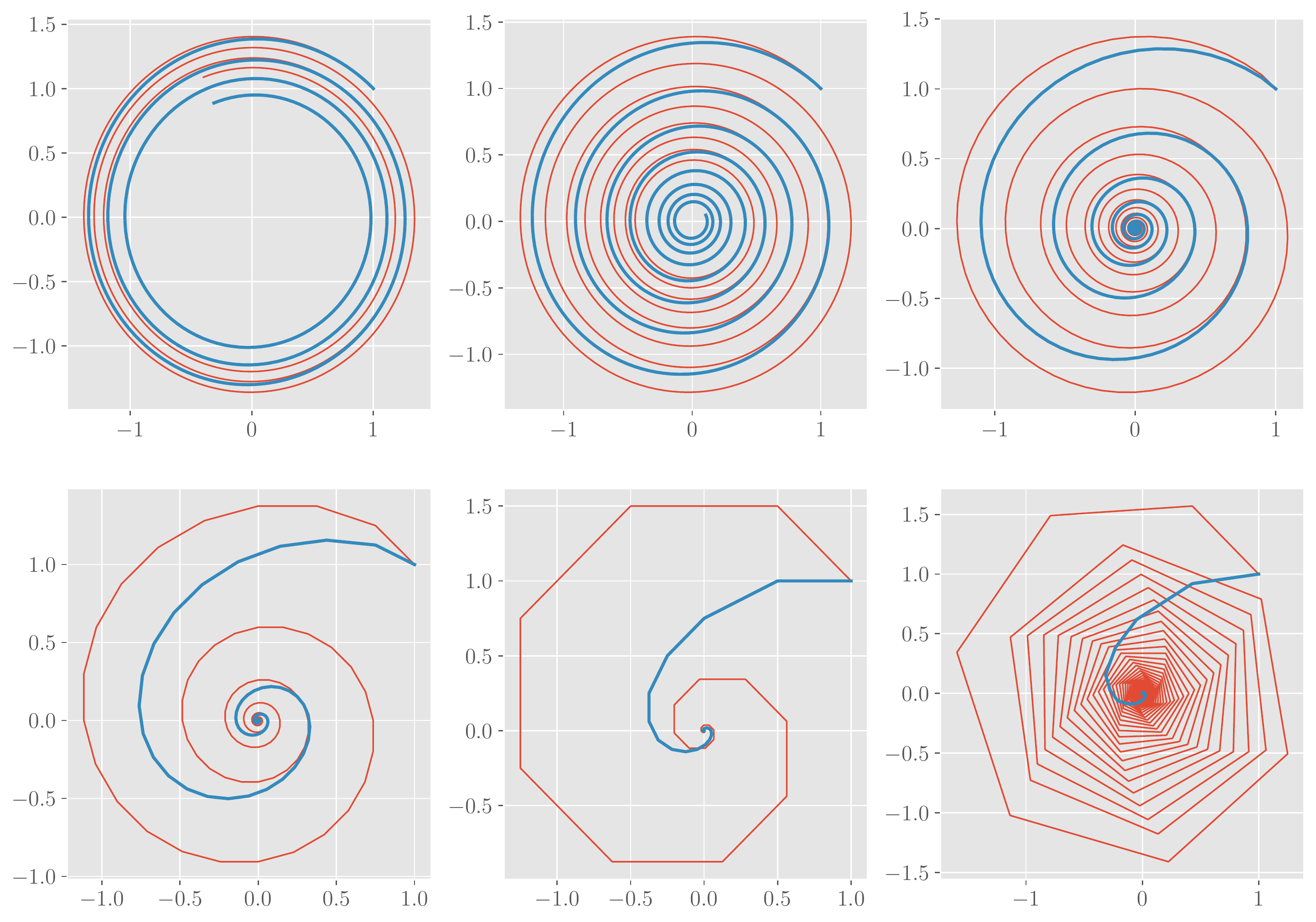}
    \caption{Alternating vs simultaneous dynamics for the one-dimensional objective function $x^{\top} y$ under different learning rates. In particular, these trajectories correspond to $\eta \in \{0.02, 0.05, 0.1, 0.25, 0.5, 0.57 \}$, from top-left, top-middle, until bottom-right respectively. The overall conclusion is that the alternating dynamics converge considerably faster to the equilibrium point.}
    \label{fig:exp-alt}
\end{figure}

\subsection{Noise in the Operator}
\label{section:noise}

Another important advantage of the currently employed framework is that it easily allows to extend the characterization under some noise in the observed value of the operator. In particular, in this section we will posit the \emph{relative deterministic noise} model, wherein instead of observing the value $F(x_k)$, the optimization algorithm has access to $F(x_k) + r_k$, where $||r_k|| \leq \delta || F(x_k)||$ for some noise parameter $\delta > 0$. That is, the magnitude of the noise depends on the norm of the operator, evaluated at some given point. An important feature of this model is that the noise may be completely \emph{adversarial}, subject to satisfying the previous constraint.

We will explain how this particular noise model can be very naturally incorporated within our framework. Specifically, first observe that the noisy observation can be seen as a cascade of the nonlinearity $P$ and an operator $\Delta$ which maps the output $v_k$ of $P$ to a signal $s$ such that $s_k = v_k + r_k$, in turn implying that $||s_k - v_k|| = ||r_k|| \leq \delta ||v_k||$. Thus, we can easily derive the gain of the noisy nonlinearity $P_{\Delta}'$ (after applying the linear shift transform):

\begin{restatable}{claim}{gainoisy}
    \label{claim:gain_noisy}
    The gain of the noisy nonlinearity $P_{\Delta}'$ is at most $(L - \mu)/2 + L \delta$.
\end{restatable}

See \Cref{appendix:proof_gain_noisy} for the proof. We should note that the bound on the noiseless gain is $(L - \mu)/2$, and as a result, the presence of $\delta$-noise may slightly increase the gain of the nonlinearity; in the context of the circle criterion, observe that as $\delta$ increases the circle gradually shrinks, but importantly, the controller remains invariant. Indeed, our previous analysis regarding the stability and the gain of each controller can be invoked in order to establish the following results:

\begin{theorem}[Gradient Descent under Noise]
    Consider an operator $F$ which satisfies \Cref{assumption:sector} and \Cref{assumption:unique}. Moreover, assume that the observed operator is corrupted with $\delta$-noise such that $\delta < \mu/L$. Then, if $\eta = 1/L$, the Gradient Descent method converges linearly to the fixed point of $F$ with rate $\rho$, for any $\rho > 1 - \kappa^{-1} + \delta$.
\end{theorem}

\begin{theorem}[Optimistic Gradient Descent under Noise]
    Consider an operator $F$ which satisfies \Cref{assumption:sector} and \Cref{assumption:unique}. Moreover, assume that the observed operator is corrupted with $\delta$-noise such that $\delta \leq \mu/(3L)$. Then, if $\eta = 1/(2L)$, the Optimistic Gradient Descent method converges linearly to the fixed point of $F$ with rate $\rho$, for any $\rho > 1 - \kappa^{-1}/4$.
\end{theorem}

We omit the proofs of these theorems since they follow similarly to our previous results. Naturally, analogous results apply for the generalized $\ogd$ method along the regime of parameters previously considered. 

\section{Future Directions}

In conclusion, we have employed a simple and robust frequency-domain framework for analyzing ``historical'' gradient-based methods. In terms of future directions, we would be interested to see whether the proposed framework can be extended or modified for the following settings:

\begin{itemize}
    \item ``Two-step'' processes such as Korpelevich's Extra-Gradient method \cite{Korpelevich1976TheEM}, and ``asymmetric'' dynamics such as the alternating updates (beyond linear landscapes).
    \item The boundary of monotonicity; for $\mu = 0$, \Cref{theorem:main_theorem} (and subsequently the small gain theorem) are not applicable. Although one could use an $\epsilon$-monotone regularizer and then take the limit $\epsilon \to 0$ (as in \cite{DBLP:journals/siamjo/LessardRP16}), it is unclear how such an approach can distinguish between $\gd$ and $\ogd$ in the context of zero-sum games.
    \item Time-varying processes: For example, can we provide a characterization for $\eta(t) = \eta_0/t$?
    \item Beyond first-order methods: Can we analyze algorithms such as Newton's method?
\end{itemize}

\paragraph{Acknowledgments.} We are grateful to the anonymous reviewers at SOSA for insightful comments and suggestions. We also thank Guodong Zhang for useful pointers in the literature.

\bibliography{./paper.bib}

\appendix

\section{Background on Nonlinear Control Theory}
\label{section:background}

Let $\ell_{2e}^p$ be the space of sequences $(x_0, x_1, \dots)$ such that $x_i \in \mathbb{R}^p$; the superscript will be typically omitted as it will be clear from the context. For a \emph{square-summable} signal $u \in \ell_2 \subseteq \ell_{2e}$, the $\ell_2$-norm $|| \cdot ||_{\ell_2}$ is defined as

\begin{equation}
    ||u||_{\ell_2} = \sqrt{\sum_{i=0}^{\infty} u_i^{\top} u_i} < \infty.
\end{equation}

We will say that a mapping (or interchangeably a \emph{system}) $H : \ell_{2e}^p \to \ell_{2e}^q$ is \emph{finite-gain} $\ell_2$-stable if there exists a (non-negative) constant $\gamma$ such that 

\begin{equation}
\label{eq:finite_gain}
    || H u ||_{\ell_2} \leq \gamma || u ||_{\ell_2},    
\end{equation}
for all $u \in \ell_2$, where $Hu$ represents the output of the system $H$ under the input $u$. The property of finite-gain stability will always be implied in the $\ell_2$-norm throughout this work, while for notational convenience we will simply (overload) write $||u ||$ to denote the $\ell_2$-norm of a signal $u$. It should be noted that---in the literature of control theory---the definition of finite-gain stability sometimes incorporates a \emph{bias} term $b \geq 0$ in \eqref{eq:finite_gain} (see \cite{Khalil:1173048}). Notice that under our definition a finite-gain stable operator $H$ has to map zero inputs to zero outputs. The minimal $\gamma \geq 0$ which satisfies \eqref{eq:finite_gain}, for all $u \in \ell_2$, will be referred to as the \emph{gain} of the operator $H$, and it will be denoted with $||H||$, assuming that such finite $\gamma < \infty$ exists.

\subsection{Quadratic Sector Bound}

A \emph{quadratic sector bound} (QSB) for a system with input/output pair $(u, v)$ (respectively) is a constraint of the form

\begin{equation}
    \label{eq:qsb-general}
    \langle w_k, \mathbf{Q} w_k \rangle \geq 0, \quad \forall k \geq 0,
\end{equation}
where $w_k^{\top} = [u_k^{\top}, v_k^{\top}]$, and $\mathbf{Q}$ is some fixed symmetric matrix. We will say that a system lies inside the sector $\mathcal{S}(\mathbf{Q})$ when all of its input/output trajectories satisfy \eqref{eq:qsb-general}. Different choices of the $\mathbf{Q}$ matrix capture different properties of dynamical systems, specifying the sector geometry; see \cite{8263814}. In an optimization context, a $\mu$-monotone operator $F$ is associated with the sector $[\mu, + \infty]$, while a co-coercive operator with parameter $1/L$ is associated with the sector $[0, L]$. Naturally, a system which satisfies both $[\mu, \infty]$ and $[0, L]$ sectors lies in the sector $[\mu, L]$. For additional details we refer the interested reader to \cite{Khalil:1173048}.

To be more concrete, let us relate the concept of a (quadratic) sector bound (QSB)\footnote{A sector bound will always refer to a quadratic sector bound in this work.} with standard concepts in convex optimization. The upshot is that the usual convexity and smoothness assumptions can be cast as a particular QSB. To be more precise, let us first recall the following definitions:

\begin{definition}[Smoothness]
A continuously differentiable function $f: \mathbb{R}^d \to \mathbb{R}$ is $L$-smooth if for any $x, x' \in \mathbb{R}^d$, 

\begin{equation}
    ||\nabla f(x) - \nabla f(x') || \leq L || x - x'||.
\end{equation}
\end{definition}

\begin{definition}[Convexity]
A continuously differentiable function $f: \mathbb{R}^d \to \mathbb{R}$ is $\mu$-strongly convex if for any $x, x' \in \mathbb{R}^d$,

\begin{equation}
    f(x) \geq f(x') + \nabla f(x')^{\top} (x - x') + \frac{\mu}{2} || x - x'||^2.
\end{equation}
\end{definition}

Notice that for $\mu = 0$ this definition recovers the usual definition of convexity. Now let us denote with $\mathcal{F}(\mu, L)$ the set of continuously differentiable functions which are both $L$-smooth and $\mu$-strongly convex. For a function $f: \mathbb{R}^d \to \mathbb{R}$ such that  $f \in \mathcal{F}(\mu, L)$ it follows that for any $x, x' \in \mathbb{R}^d$, 

\begin{equation*}
    \mu ||x - x'||^2 \leq (\nabla f(x) - \nabla f(x'))^{\top} (x - x') \leq L || x - x'||^2.
\end{equation*}

The ratio $\kappa \define L/\mu \geq 1$ is called the \emph{condition ratio} of $f$; this terminology is used to distinguish the condition ratio of a function from the related notion of the \emph{condition number} of a matrix.\footnote{The connection is that if $f$ is twice differentiable, $\text{cond}(\nabla^2 f(x)) \leq \kappa$, for all $x \in \mathbb{R}^d$.} The following lemma reveals the connection with the concept of a QSB:

\begin{lemma}[QSB for Smooth-Convex Functions, \cite{DBLP:conf/cdc/BoczarLR15}]
\label{lemma:QSB}
Consider a continuously differentiable function $f : \mathbb{R}^d \to \mathbb{R}$ such that $f \in \mathcal{F}(\mu, L)$. Then, for any $x, x' \in \mathbb{R}^d$ the following property holds:

\begin{equation}
    \label{eq:QSB}
    \begin{bmatrix}
    x - x' \\
    \nabla f(x) - \nabla f(x')
    \end{bmatrix}^{\top}
    \begin{bmatrix}
    -2 \mu L \mathbf{I}_d & (L + \mu) \mathbf{I}_d \\
    (L + \mu) \mathbf{I}_d & -2\mathbf{I}_d
    \end{bmatrix}
    \begin{bmatrix}
    x - x'\\
    \nabla f(x) - \nabla f(x')
    \end{bmatrix}
    \geq 0.
\end{equation}
\end{lemma}

\begin{proof}
First, let $g(x)$ be some convex and $L$-smooth continuously differentiable function. It is well-known that for all $x, x' \in \mathbb{R}^d$,

\begin{equation*}
    \label{eq:co-coercivity}
    (\nabla g(x) - \nabla g(x'))^{\top} (x - x') \geq \frac{1}{L} || \nabla g(x) - \nabla g(x') ||^2;
\end{equation*}
that is, $\nabla g$ is co-coercive with parameter $1/L$. In particular, let $g(x) \define f(x) - \frac{\mu}{2} x^{\top} x$, where $f \in \mathcal{F}(\mu, L)$. Then, it follows that $g$ is convex and $(L-\mu)$-Lipschitz continuous, and hence,  the co-coercivity property (\Cref{eq:co-coercivity}) implies that 

\begin{equation*}
    (\nabla f(x) - \nabla f(x'))^{\top} (x - x') - \mu || x - x'||^2 \geq \frac{1}{L - \mu} || (\nabla f(x) - \nabla f(x') - \mu ( x - x') ||^2.
\end{equation*}
Finally, rearranging this equation gives \eqref{eq:QSB}, concluding the proof.
\end{proof}

Let us denote with $\mathcal{S}(\mu, L)$ the set of continuously differentiable functions satisfying \eqref{eq:QSB}. \Cref{lemma:QSB} implies that this class of functions includes as a subset the class $\mathcal{F}(\mu, L)$, but it is important to stress that this inclusion is strict \cite{hu2017control}. For example, $\mathcal{S}(\mu, L)$ includes smooth and \emph{restricted strongly convex} (RSC) functions \cite{karimi2020linear}.

\subsection{Feedback Interconnection Loop}

The \emph{feedback interconnection} loop with external inputs $r \in \ell_{2e}^{p}$ and $e \in \ell_{2e}^q$ is depicted in \Cref{fig:feedback_interconnection}; it can be described with the following equations: 

\[ \begin{cases} 
      v = G_1 u + e; \\
      u = G_2 v + r,
   \end{cases}
\]
where $G_1: \ell_{2e}^p \to \ell_{2e}^{q}$ and $G_2: \ell_{2e}^{q} \to \ell_{2e}^{p}$. The feedback interconnection of \Cref{fig:feedback_interconnection} is said to be \emph{well-posed} if for every pair of inputs $(r, e)$ there exist unique outputs $u, v \in \ell_{2e}$. Moreover, the interconnection is said to be finite-gain stable if there exists a constant $\gamma \geq 0$ such that

\begin{equation}
    ||u|| + ||v|| \leq \gamma (||r|| + ||e|| ),
\end{equation}
for any square-summable inputs $r$ and $e$. Most of our results are based on the following central theorem: 

\begin{figure}[!ht]
    \centering
    \includegraphics[scale=0.6]{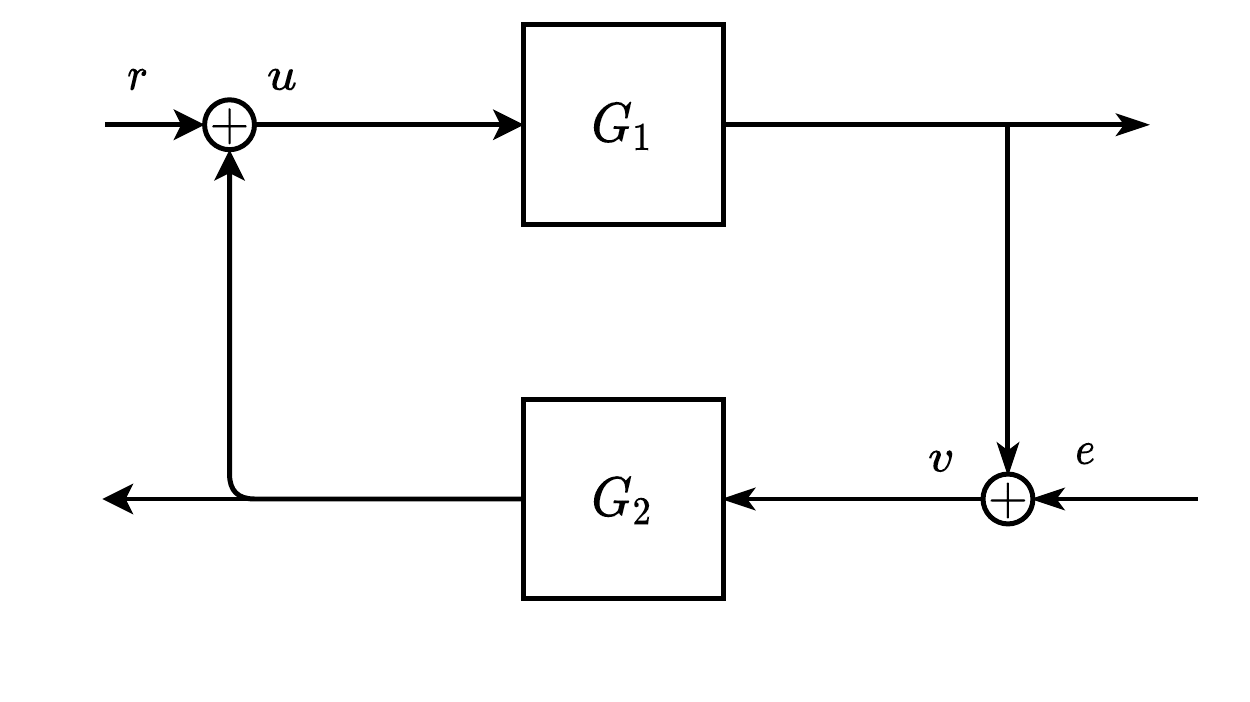}
    \caption{A feedback interconnection loop with external inputs $(r, e)$.}
    \label{fig:feedback_interconnection}
\end{figure}

\begin{theorem}[Small Gain Theorem, \cite{Zames1966OnTI}; see also \cite{88c5607c1675458893365b28433cb735,Mareels,536496}]
    \label{theorem-sgt}
Let $G_1$ and $G_2$ be finite-gain stable operators such that the feedback interconnection $[G_1, G_2]$ is well-posed. If $||G_1|| ||G_2|| < 1$, then the feedback interconnection $[G_1, G_2]$ is finite-gain stable.
\end{theorem}

\subsection{Linear Shift Transformation}
\label{appendix:lst}

However, it is often the case that the systems comprising the feedback loop fail to satisfy the hypotheses of the small gain theorem (\Cref{theorem-sgt}). As it happens, in many such cases the stability of the closed loop can be established based on a modified form of the feedback interconnection system which shares the same stability properties. Specifically, for our purposes we will employ the standard \emph{linear shift transformation}, illustrated in \Cref{fig:transformation}. \Cref{lemma:linear_shift} implies the equivalence between the two formulations.

\begin{figure}[!ht]
    \centering
    \includegraphics[scale=0.6]{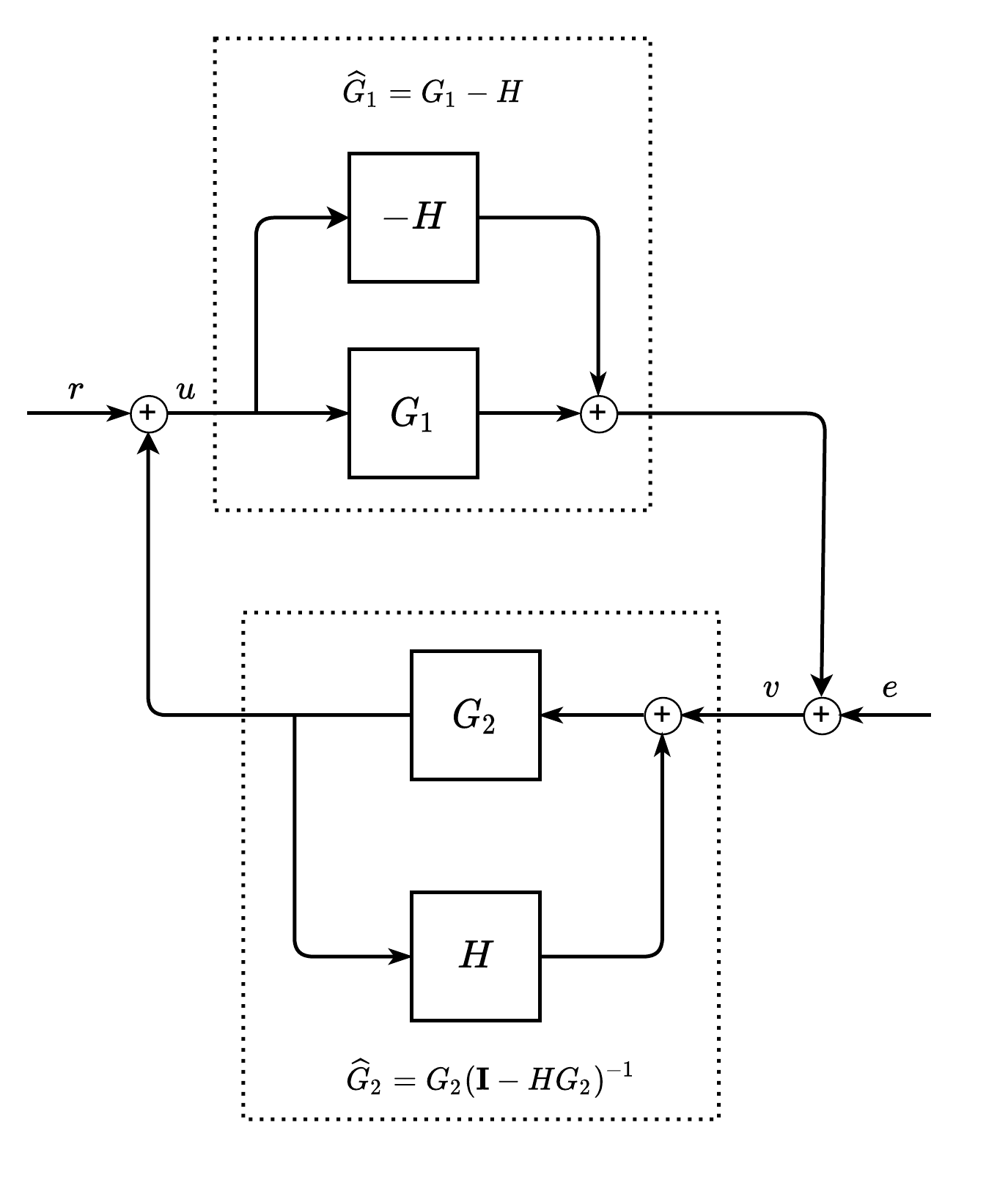}
    \caption{The linear shift transformation.}
    \label{fig:transformation}
\end{figure}

\begin{lemma}[\cite{10.5555/191301}, Lemma 3.5.3]
    \label{lemma:linear_shift}
Consider the feedback loops shown in \Cref{fig:feedback_interconnection} and \Cref{fig:transformation}, with $G_1, G_2 : \ell_{2e} \to \ell_{2e}$ and $H : \ell_{2e} \to \ell_{2e}$ being a linear system such that the operator $G_2(\mathbf{I} - H G_2)^{-1} : \ell_{2e} \to \ell_{2e}$ is well-defined. If $H$ and $G_2(\mathbf{I} - H G_2)^{-1}$ are finite-gain stable, then the loop in \Cref{fig:feedback_interconnection} is finite-gain stable if and only if the loop in \Cref{fig:transformation} is finite-gain stable.
\end{lemma}

\subsection{Exponential Stability}

In the sequel we will consider the feedback interconnection $[P, K]$ of \Cref{fig:1}, such that 

\begin{itemize}
    \item $P : \ell_{2e} \to \ell_{2e}$ represents a finite-gain stable \emph{static nonlinearity}, associated with the operator $F$;
    \item $K$ is a \emph{linear time-invariant} (LTI) system with transfer function $\mathbf{K}(z)$, associated with the optimization algorithm.
\end{itemize}

This formulation constitutes arguably the most well-studied class of feedback interconnection loops in nonlinear control theory, and deriving sufficient conditions for the absolute stability of such systems is referred to as \emph{Lur'e problem}. Indeed, we will explain how this particular formulation suffices in order to characterize the behavior of many well-studied optimization algorithms. In this context, we will say that the induced interconnection $[P, K]$ is \emph{exponentially stable} if there exists some $\rho \in (0,1)$ such that if $r = 0$ and $e = 0$, the \emph{state} $\xi_k$ of $K$ will decay exponentially with rate $\rho$; that is, 

\begin{equation}
    || \xi_k || \leq C \rho^{k} ||\xi_0||, 
\end{equation}
under any initial state $\xi_0$ and a constant $C \geq 0$ independent on $k$. We will use the following standard fact regarding the well-posedness of the induced feedback interconnection loop:

\begin{fact}
    \label{fact:well_posed}
If the transfer matrix $\mathbf{K}(z)$ is strictly proper, then the feedback interconnection $[P, K]$ is well-posed, i.e. the feedback interconnection has a well-defined state model.
\end{fact}

Recall that a rational transfer function is called \emph{strictly proper} if the degree of the numerator is (strictly) less than the degree of the denominator. That is, $\mathbf{K}(z)$ is strictly proper if $\mathbf{K}(\infty) = \mathbf{0}$. Equivalently, if $(\mathbf{A}, \mathbf{B}, \mathbf{C}, \mathbf{D})$ represent the \emph{state matrices} of $K$ (see \Cref{subsection:dynamical_systems}), the feedback interconnection $[P, K]$ is well-posed if $\mathbf{D} = \mathbf{0}$.

We will use a very elegant result due to Boczar, Lessard, and Recht \cite{DBLP:conf/cdc/BoczarLR15} which essentially reduces the exponential stability of a feedback interconnection to the finite-gain stability of a ``gap-introducing'' transformed system. More precisely, let us define the operators $\rho_{+}$ and $\rho_{-}$ as the time-domain multipliers $\rho^{k}$ and $\rho^{-k}$ respectively, for some fixed parameter $\rho \in (0,1)$; these operators derive from the theory of \emph{stability multipliers} \cite{877001}. The following result provides a sufficient condition for certifying exponential stability:

\begin{lemma}[\cite{DBLP:conf/cdc/BoczarLR15}, Proposition 5]
    \label{lemma:exponential_reduction}
If the interconnection system of \Cref{fig:2} is finite-gain stable, then the interconnection system of \Cref{fig:1} is exponentially stable.
\end{lemma}

\begin{figure}[!ht]
\centering
\begin{minipage}{.5\textwidth}
  \centering
  \includegraphics[scale=0.6]{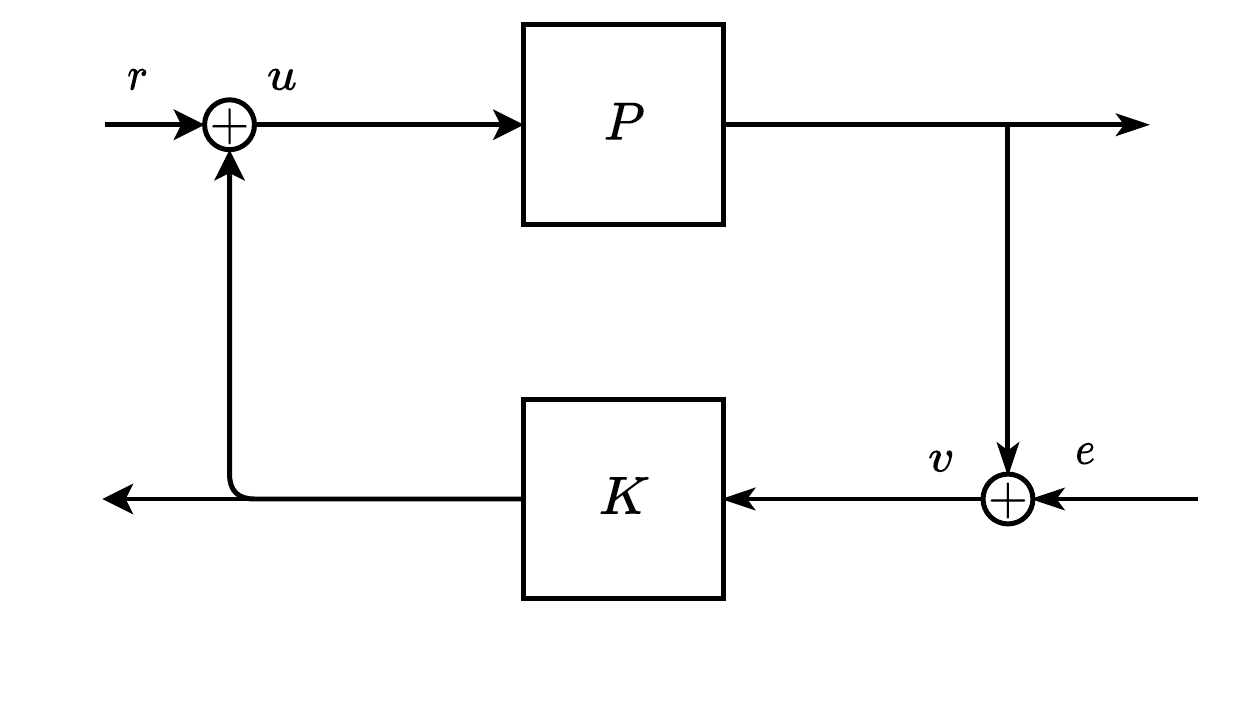}
  \captionof{figure}{The $[P, K]$ interconnection.}
  \label{fig:1}
\end{minipage}%
\begin{minipage}{.5\textwidth}
  \centering
  \includegraphics[scale=0.6]{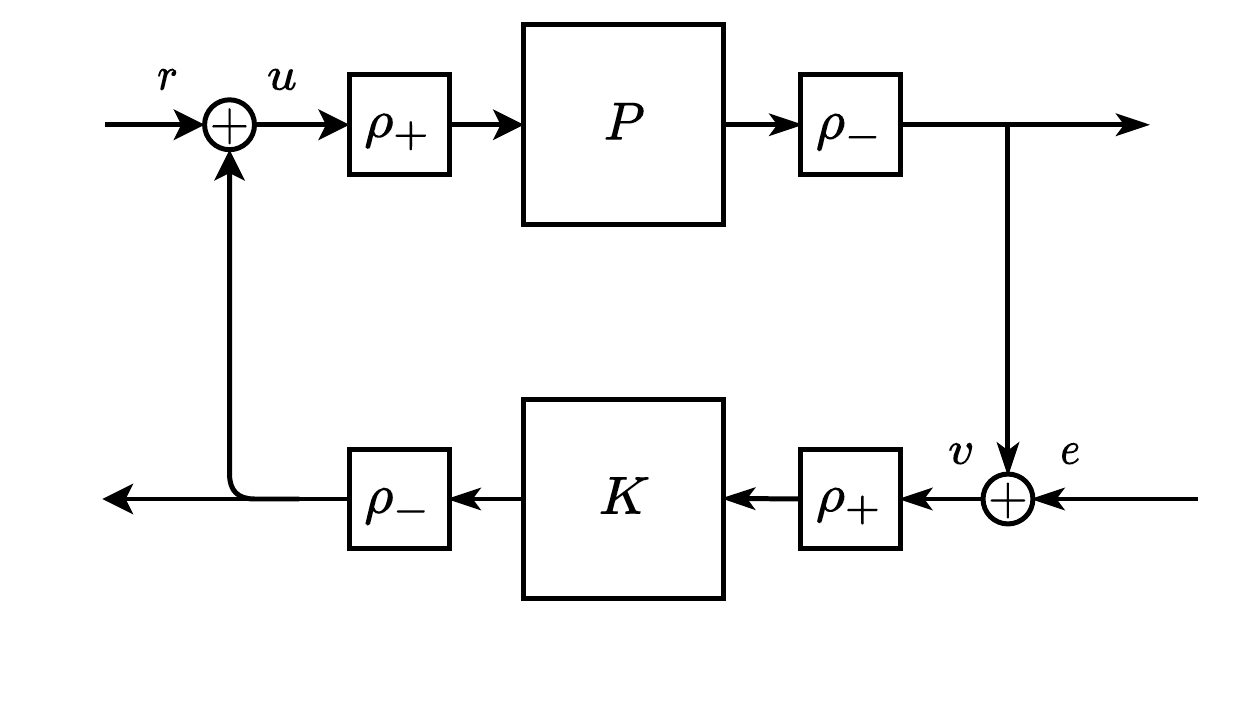}
  \captionof{figure}{The transformed interconnection.}
  \label{fig:2}
\end{minipage}
\end{figure}

We remark that a slightly different definition of stability is used in \cite{DBLP:conf/cdc/BoczarLR15}, but this does not alter the result. Therefore, one way to establish exponential stability is to apply the small gain theorem, but for the transformed system of \Cref{fig:2}. To this end, we will use the following basic observations:

\begin{claim}[\cite{DBLP:conf/cdc/BoczarLR15}, Remark 4]
    \label{claim:K_r}
If $\mathbf{K}(z)$ is the transfer matrix of $K$, then the system $\rho_+ \circ K \circ \rho_-$ has transfer matrix $\mathbf{K}(\rho z)$.
\end{claim}

This claim follows by simply applying the z-transform, and using the property $\mathcal{Z} \{ u_k \rho^k \} (z) = \mathcal{Z} \{u_k \} (z/\rho)$. We should note that the notation $\circ$ stands for the usual composition of systems, such that $(f \circ g) (u) \equiv g(f(u))$, for a signal $u$.

\begin{claim}
    \label{claim:transformed-sector_bound}
If the operator $P$ lies in the sector $\mathcal{S}(\mathbf{Q})$, for some symmetric matrix $\mathbf{Q}$, then the operator $P_{\rho} = \rho_+ \circ P \circ \rho_-$ also lies in $\mathcal{S}(\mathbf{Q})$.
\end{claim}

\begin{proof}
Let us consider the notation of \Cref{fig:transf_P}. By assumption, we can infer that for all $k \geq 0$,

\begin{equation*}
    \label{eq:P-sector}
    \begin{bmatrix}
    (u_k')^{\top} & (v_k')^{\top}
    \end{bmatrix}
    \mathbf{Q}
    \begin{bmatrix}
    u_k' \\
    v_k'
    \end{bmatrix} \geq 0.
\end{equation*}
But, we also know that $u_k' = \rho^k u_k$, and $v_k = \rho^{-k} v_k' \iff v_k' = \rho^k v_k$. Thus, substituting in \eqref{eq:P-sector} leads to the desired conclusion.
\end{proof}

\begin{figure}[!ht]
    \centering
    \includegraphics[scale=0.7]{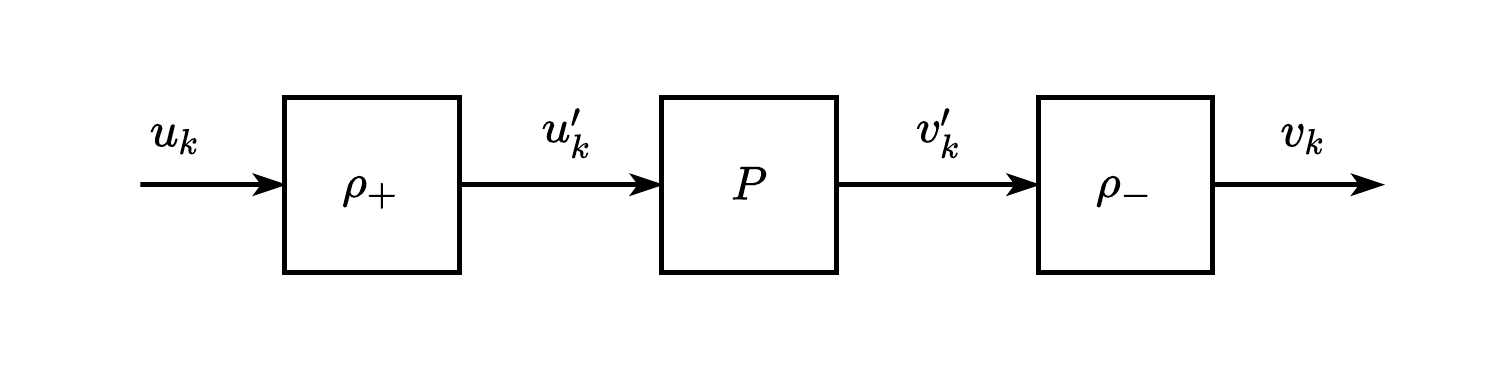}
    \caption{The operator $P_{\rho} = \rho_+ \circ P \circ \rho_-$.}
    \label{fig:transf_P}
\end{figure}

\subsection{Bounding the Gains}

In this subsection we explain how one can bound the gains of the operators involved in the feedback interconnection, which is essentially the crux in applying the small gain theorem (\Cref{theorem-sgt}). In particular, first let $P$ be defined with the input/output relation $v_k = F(u_k + x^*)$, where $F$ is an operator satisfying \Cref{assumption:sector} and \Cref{assumption:unique}. Then, we know that

\begin{equation}
    \label{eq:non_diagonal-SB}
    \begin{bmatrix}
    u_k \\
    v_k
    \end{bmatrix}^{\top}
    \begin{bmatrix}
    -2 \mu L \mathbf{I}_d & (L + \mu) \mathbf{I}_d \\
    (L + \mu) \mathbf{I}_d & -2\mathbf{I}_d
    \end{bmatrix}
    \begin{bmatrix}
    u_k \\
    v_k
    \end{bmatrix}
    \geq 0.
\end{equation}

That is, $P$ lies in the sector $[\mu, L]$. We will show how the linear shift transformation (\Cref{fig:transformation}) previously introduced can \emph{diagonalize} this QSB. First, we let $H$ be the operator which simply multiplies the input with a scalar value $h$; observe that $H$ is indeed linear and stable, as required for \Cref{lemma:linear_shift}. For $h = (L + \mu)/2$ we will show that the induced sector bound is diagonal, which immediately implies a bound for the gain.

\begin{lemma}
    \label{lemma:P-gain}
    The system $P' = P - (L + \mu)/2$ satisfies the following diagonal sector bound:
    
    \begin{equation}
        \label{eq:diagonal_SB}
    \begin{bmatrix}
    u_k \\
    v_k
    \end{bmatrix}^{\top}
    \begin{bmatrix}
    \frac{(L - \mu)^2}{2} \mathbf{I}_d & \mathbf{0}_d \\
     \mathbf{0}_d & -2\mathbf{I}_d
    \end{bmatrix}
    \begin{bmatrix}
    u_k \\
    v_k
    \end{bmatrix}
    \geq 0,
\end{equation}
where $(u, v)$ represents the input/output pair (respectively) of $P'$. In particular, this implies that $||P'|| \leq (L - \mu)/2$.
\end{lemma}

\begin{proof}
Let $v_k' = P u_k$ and $v_k = P' u_k = v_k' - h u_k$, where $h = (L+ \mu)/2$. We know that the input/output pair $(u, v')$ of $P$ satisfies the quadratic sector bound of \Cref{eq:non_diagonal-SB}, while the structure of the linear shift transformation implies the following linear transformation:

\begin{equation*}
    \begin{bmatrix}
    u_k \\
    v_k' 
    \end{bmatrix}
    = 
    \begin{bmatrix}
    \mathbf{I}_d & \mathbf{0}_d \\
    h \mathbf{I}_d & \mathbf{I}_d 
    \end{bmatrix}
    \begin{bmatrix}
    u_k \\
    v_k
    \end{bmatrix}.
\end{equation*}
Thus, we obtain that 

\begin{equation*}
    \begin{bmatrix}
    u_k \\
    v_k 
    \end{bmatrix}^{\top}
    \begin{bmatrix}
    \mathbf{I}_d & \frac{L + \mu}{2} \mathbf{I}_d \\
    \mathbf{0}_d & \mathbf{I}_d 
    \end{bmatrix}
    \begin{bmatrix}
    -2 \mu L \mathbf{I}_d & (L + \mu) \mathbf{I}_d \\
    (L + \mu) \mathbf{I}_d & -2\mathbf{I}_d
    \end{bmatrix}
    \begin{bmatrix}
    \mathbf{I}_d & \mathbf{0}_d \\
    \frac{L + \mu}{2} \mathbf{I}_d & \mathbf{I}_d 
    \end{bmatrix}
    \begin{bmatrix}
    u_k \\
    v_k
    \end{bmatrix}
    \geq 0.
\end{equation*}
As a result, performing the matrix multiplications yields the claimed diagonal sector bound. Finally, \Cref{eq:diagonal_SB} implies that

\begin{equation}
    \frac{(L - \mu)^2}{2} || u_k ||^2 - 2 || v_k ||^2 \geq 0 \iff || v_k ||^2 \leq \left( \frac{L - \mu}{2} \right)^2 || u_k ||^2,
\end{equation}
for all $k \geq 0$. Thus, summing over all $k$ (and assuming that $u \in \ell_2$) yields that $||P|| \leq (L - \mu)/2$, concluding the proof.
\end{proof}

\begin{remark}
It is not hard to show that for standard convex optimization the gain of the nonlinearity is in fact \emph{strictly} less than $(L - \mu)/2$. This implies that (unsurprisingly) all of the analyzed optimization methods will remain stable even in the boundary of convexity $\mu = 0$---although the convergence is clearly not linear in this case. However, it is well-know by now that this is not the case for min-max optimization \cite{DBLP:conf/iclr/DaskalakisISZ18}.
\end{remark}

The nonlinearity $P$---and hence its gain---will remain invariant across the analysis of different algorithms; the structure of each optimization method will be captured through the LTI system $K$. In particular, we will use the following fundamental result, which can be found in any textbook on robust control theory: 

\Hinfinity*

\subsection{Exponential Convergence of Optimization Algorithms}

Finally, the previous ingredients are combined in order to provide a sufficient condition for the exponential convergence of an optimization algorithm, described with the LTI system $K$:

\maintheorem*

\begin{proof}
Consider the linear feedback transformation $[P_{\rho}', K_{\rho}']$ of \Cref{fig:2}, for some $\rho \in (0,1)$, with $P_{\rho}' = P_{\rho} - h$ and $\mathbf{K}_{\rho}'(z) = \mathbf{K}'(\rho z)$ (\Cref{claim:K_r}). \Cref{lemma:P-gain} and \Cref{claim:transformed-sector_bound} imply that $||P_{\rho}'|| \leq (L - \mu)/2$. Moreover, by assumption it follows that $||P_{\rho}'|| || K_{\rho}'|| < 1$; also notice that the feedback interconnection $[P_{\rho}', K_{\rho}']$ is well-posed given that by assumption the controller $\mathbf{K}(z)$ (and hence $\mathbf{K}'(\rho z)$) is strictly proper. Thus, the small gain theorem implies that the feedback interconnection $[P_{\rho}', K_{\rho}']$ is finite-gain stable. By \Cref{lemma:linear_shift} this also implies that $[P_{\rho}, K_{\rho}]$ is finite-gain stable; notice that the operator $H$ is trivially linear and finite-gain stable. As a result, we can apply \Cref{lemma:exponential_reduction} to deduce that $|| x_k - x^*|| \leq C \rho^{k} || x_0 ||$, for any initial state $x_0 \in \mathbb{R}^d$. This concludes the proof.
\end{proof}

We should note that the matrix $\mathbf{K}'(z)$ is usually referred to as the \emph{complementary sensitivity matrix} of $\mathbf{K}(z)$.

\subsection{Linear Dynamical Systems}
\label{subsection:dynamical_systems}

A linear dynamical system is a set of recursive linear equations of the following form:

\begin{subequations}
\label{eq:dynamical_system}
\begin{align}
    \xi_{k+1} = \mathbf{A} \xi_k + \mathbf{B} v_k; \\
    u_k = \mathbf{C} \xi_k + \mathbf{D} v_k.
\end{align}
\end{subequations}

At every time-step $k = 0, 1, \dots$, $v_k$ represents the input, $u_k$ the output, while $\xi_k$ corresponds to the \emph{state} of the system. The linear dynamical system described with \eqref{eq:dynamical_system} is usually expressed more succinctly in the following block notation:

\begin{equation}
    \left[ 
\begin{array}{c|c} 
  \mathbf{A} & \mathbf{B} \\ 
  \hline 
  \mathbf{C} & \mathbf{D} 
\end{array}
\right].
\end{equation}

Importantly, first-order optimization algorithms, such as Gradient Descent, can be expressed in the form of \Cref{eq:dynamical_system}, where the input signal $v$ is associated, for example, with the min-max gradients of the underlying objective function.

\subsection{The Circle Criterion}

An alternative way of analyzing the stability of the feedback interconnection loop, besides the small gain theorem, is the so-called \emph{circle criterion}, which also offers a way of visualizing different optimization methods. Specifically, let us consider the stability of the standard feedback interconnection system of \Cref{fig:feedback_interconnection} with zero external inputs, i.e. the \emph{unforced} system. We will say that the system is \emph{absolutely stable} if it has a globally uniformly asymptotically stable equilibrium point at the origin for all nonlinearities within a given sector; in turn, this would imply that the unique fixed point of $F$ is a global attractor of the corresponding dynamics. In this context, the circle criterion offers a frequency-domain sufficient condition for absolute stability. In the sequel we will use the one-dimensional criterion; this is justified since although we are studying multi-dimensional systems, their transfer matrix can be expressed as $\mathbf{K}(z) = K(z) \mathbf{I}_d$.

\begin{proposition}[\cite{Khalil:1173048}]
    Consider a nonlinearity which lies in the sector $[a,b]$, such that $a < 0 < b$. Then, the system is absolutely stable if $K'(z)$ is stable and the Nyquist plot of $K'(e^{j \omega})$ lies in the interior of the disk $D(a, b)$
\end{proposition}

\section{Omitted Proofs}

In this section we provide all of the proofs omitted from \Cref{section:analysis}.

\subsection{Proof of \texorpdfstring{\Cref{proposition:OGD-tightness}}{}}
\label{appendix:proof-tightness}

\tightness*

\begin{proof}
First, notice that the OGD dynamics can be expressed as the following dynamical system:

\begin{equation*}
    \begin{bmatrix}
    x_{k+1} \\
    x_k
    \end{bmatrix}
    =
    T \left(
    \begin{bmatrix}
    x_k \\
    x_{k-1}
    \end{bmatrix}
    \right)
    = 
    \begin{bmatrix}
    x_k - 2\eta F(x_k) + \eta F(x_{k-1}) \\
    x_k 
    \end{bmatrix},
\end{equation*}
where $T: \mathbb{R}^{2d} \to \mathbb{R}^{2d}$. In particular, for $f(x) = x^2 - \cos(x)$ and $F(x) = \nabla f (x)$, it follows that $T: \mathbb{R}^2 \to \mathbb{R}^2$ is continuously differentiable, and in particular,

\begin{equation*}
    T \left(
    \begin{bmatrix}
    x_k \\
    x_{k-1}
    \end{bmatrix}
    \right)
    =
    \begin{bmatrix}
    x_k - 2\eta (2x_k + \sin x_k) + \eta (2 x_{k-1} + \sin x_{k-1}) \\
    x_k
    \end{bmatrix}.
\end{equation*}

It is easy to see that the unique critical point of the induced dynamical system arises at $(0, 0) \in \mathbb{R}^2$. Moreover, the Jacobian of mapping $T$ at the critical point---the \emph{linearization} of the dynamical system---reads 

\begin{equation*}
    \mathbf{J}_T(0, 0) = 
    \begin{bmatrix}
    1 - 6 \eta & 3\eta \\
    1 & 0
    \end{bmatrix}.
\end{equation*}

As a result, the characteristic equation of the Jacobian at the critical point is $\lambda^2 + \lambda (6\eta - 1) - 3\eta = 0$, and the corresponding eigenvalues are 

\begin{equation*}
    \lambda_{1,2} = \frac{1 - 6 \eta \pm \sqrt{36 \eta^2 + 1}}{2}.
\end{equation*}

Now observe that $\lambda_1 \in (0,1)$, while for $\eta > 2/9$ it follows that $|\lambda_2| > 1$. Thus, we conclude that the unique critical point $(0, 0)$ is unstable, and OGD diverges under any non-trivial initialization.
\end{proof}

\subsection{Proof of \texorpdfstring{\Cref{theorem:GOGD}}{}}
\label{appendix:proof-GOGD}

\genogd*

\begin{proof}
We will again denote with $\lambda = \kappa^{-1} \in (0,1)$. For $\alpha = 1/(2L)$ and $\beta = \ell/(2L)$ the complementary sensitivity function of $\gogd$ reads 

\begin{equation*}
    K'(\rho z) = \frac{1}{2L} \frac{\ell - (1 + \ell) \rho z}{\rho^2 z^2 + \frac{-3 + \lambda + \ell + \lambda \ell}{4}\rho z - \frac{\ell (1+\lambda)}{4}}.
\end{equation*}
Let $z_0, z_1$ be the roots of the characteristic equation

\begin{equation}    
    \label{eq:quadratic-GOGD}
    z^2 + \frac{-3 + \lambda + \ell + \lambda \ell}{4} z - \frac{\ell (1+\lambda)}{4} = 0.
\end{equation}

Then, it follows that the $\gogd$ controller is stable if $\rho > \max \{|z_0|, |z_1|\}$. In particular, solving \eqref{eq:quadratic-GOGD} yields 

\begin{equation*}
    z_{0, 1} = \frac{1}{8} \left( 3 - \lambda - \ell - \lambda \ell \pm \sqrt{(-3 + \lambda + \ell + \lambda \ell)^2 + 16\ell (1+\lambda)}  \right).
\end{equation*}

Observe that $z_0 > 0 \geq z_1$ and $|z_0| > |z_1|$. Next, we will use an elementary lemma to upper-bound the magnitude of $z_0$, and subsequently of $\max \{ |z_0|, |z_1|\}$.

\begin{lemma}
    For any $\ell \in [0, 1]$ and $\lambda \in (0,1)$,
    \begin{equation*}
        \sqrt{(-3 + \lambda + \ell + \lambda \ell)^2 + 16 \ell (1+\lambda)} < 5 - \lambda + \ell + \lambda \ell.
    \end{equation*}
\end{lemma}
As a result, it follows that $|z_0| = z_0 < 1 - \lambda/4$. Thus, it suffices to take $\rho = 1 - \lambda/4$ to ensure that $K'(\rho z)$ is stable. The next step is to bound the gain of the controller. To this end, \Cref{theorem:H-infinity} implies that

\begin{equation*}
    ||K'_{\rho}|| = \frac{1}{2L}  \sup_{\omega \in [-\pi, \pi]} \left| \frac{\ell - (1+\ell) e^{j \omega}}{\rho^2 e^{2j\omega} + \frac{-3 + \lambda + \ell + \lambda \ell}{4} \rho e^{j \omega} - \frac{\ell (1+\lambda)}{4}} \right|,
\end{equation*}
and with simple calculations we can see that

\begin{equation}
    \label{eq:psi-GOGD}
    \left| \frac{\ell - (1+\ell) e^{j \omega}}{\rho^2 e^{2j\omega} + \frac{-3 + \lambda + \ell + \lambda \ell}{4} \rho e^{j \omega} - \frac{\ell (1+\lambda)}{4}} \right|^2 = \frac{a_0 + a_1 \cos(\omega)}{b_0 + b_1 \cos(\omega) + b_2 \cos(2\omega)} \define \psi(\omega),
\end{equation}
where
\begin{subequations}
\begin{align*}
    a_0 &= \ell^2 + (1 + \ell)^2 \rho^2 ; \\
    a_1 &= - 2\ell (1+\ell) \rho ; \\
    b_0 &= \rho^4 + \frac{(-3 + \lambda + \ell + \lambda \ell)^2}{16} \rho^2 + \frac{\ell^2}{16}(1 + \lambda)^2; \\
    b_1 &= \left( \frac{(-3 + \lambda + \ell + \lambda \ell)\rho}{2} \right) \left( \rho^2 - \frac{\lambda (1+\ell)}{4} \right) ; \\
    b_2 &= - \frac{\ell (\lambda + 1)}{2} \rho^2.
\end{align*}
\end{subequations}

\begin{lemma}
    For $\rho = 1 - \lambda/4$, the function $\psi(\omega)$ as defined in \eqref{eq:psi-GOGD} attains its maximum either on $\omega = 0$ or $\omega = \pi$.
\end{lemma}
Thus, for $\rho = 1 - \lambda/4$ it follows that $|| K'_{\rho} || = \max \{ |K'(\rho)|, |K'(-\rho)| \}$. Finally, the theorem follows directly from the following lemma, implying that the conditions of the small gain theorem---and subsequently of \Cref{theorem:main_theorem}---are met. 

\begin{lemma}
    Let $\lambda \in (0,1)$ and $\ell \in [0,1]$. For $\rho = 1 - \lambda/4$ the following inequalities hold:
    
    \begin{equation*}
        \label{eq:GOGD-b_1}
        |\ell - (1 + \ell) \rho)| (1 - \lambda) < 4 \left| \rho^2 + \rho \frac{-3 + \lambda + \ell + \lambda \ell}{4} - \frac{\ell (1+ \lambda)}{4}\right|; 
    \end{equation*}
    
    \begin{equation*}
        \label{eq:GOGD-b_2}
        |\ell + (1 + \ell) \rho)| (1 - \lambda) < 4 \left| \rho^2 - \rho \frac{-3 + \lambda + \ell + \lambda \ell}{4} - \frac{\ell (1+ \lambda)}{4}\right|. 
    \end{equation*}
\end{lemma}
\end{proof}

\subsection{Proof of \texorpdfstring{\Cref{proposition:alt-ogd}}{}}
\label{appendix:proof-alt-ogd}

\altch*

\begin{proof}
If we transfer the alternating $\ogd$ dynamics \eqref{eq:alt-ogd} to the $z$-space we obtain that 

\begin{align*}
\begin{split}
 (z-1) X(z) = \eta \mathbf{A} (-2 + z^{-1}) Y(z); \\
 (z-1) Y(z) = \eta \mathbf{A}^{\top} (2z - 1) X(z),
\end{split}
\end{align*}
where we used the time-delay property of the $Z$-transform. Consequently, if we decouple these equations the claim follows.
\end{proof}

\subsection{Proof of \texorpdfstring{\Cref{claim:gain_noisy}}{}}
\label{appendix:proof_gain_noisy}

\gainoisy*

\begin{proof}
In \Cref{fig:noise} we illustrate the noisy nonlinearity after applying the linear shift transform (recall \Cref{appendix:lst}). Under the notation introduced in the figure, it follows that

\begin{equation*}
    ||v_k'|| = || s_k - h u_k || \leq || s_k - v_k || + || v_k - h u_k|| \leq \delta || v_k|| + || v_k - h u_k||,
\end{equation*}
where we used the triangle inequality, and the fact that $||s_k - v_k|| \leq \delta ||v_k||$. We also know from \Cref{lemma:P-gain} that $||v_k - h u_k|| \leq \frac{L - \mu}{2} || u_k||$. Finally, observe that the operator $F$ is $L$-Lipschitz, implying that $||v_k|| \leq L ||u_k||$; this concludes the proof.
\end{proof}

\begin{figure}[!ht]
    \centering
    \includegraphics[scale=0.55]{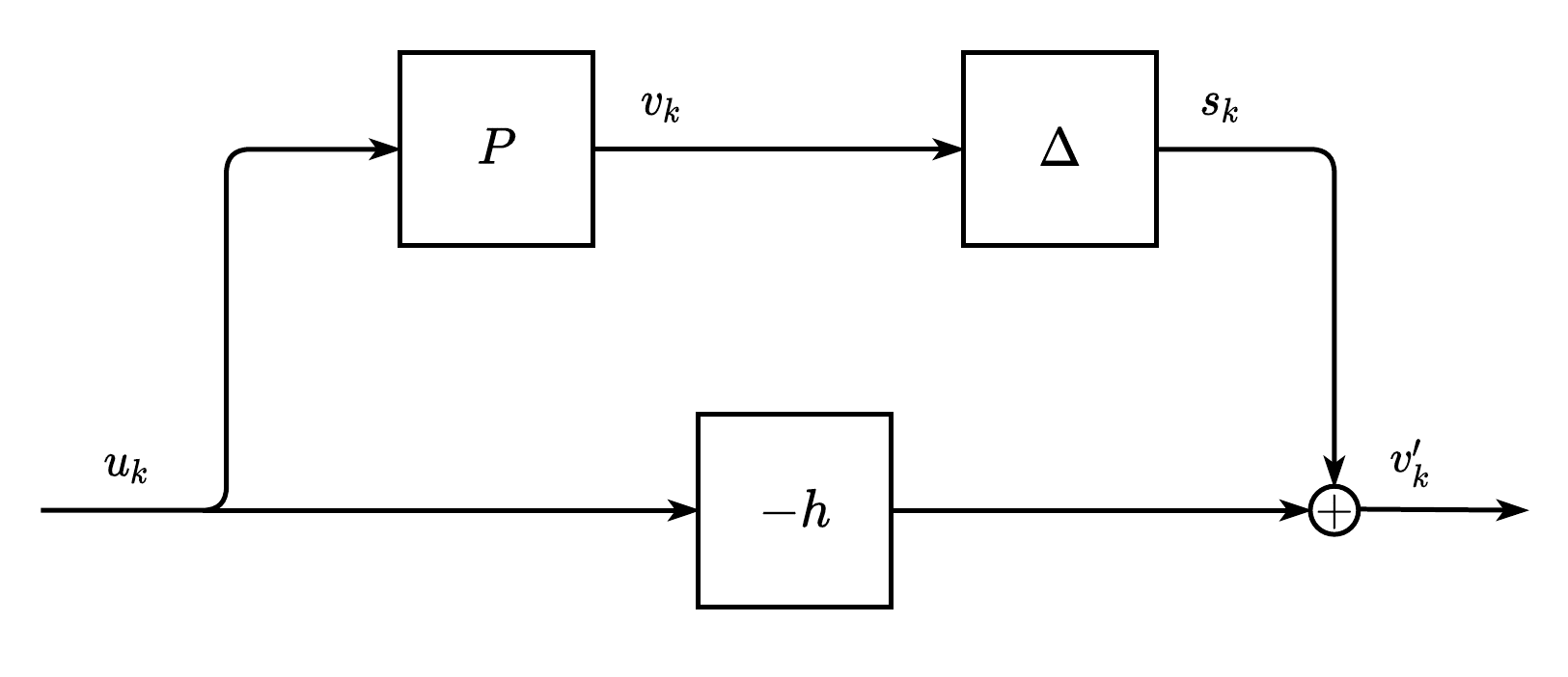}
    \caption{The ``noisy'' nonlinearity after applying the linear shift transformation.}
    \label{fig:noise}
\end{figure}

\section{Monotonicity of the Min-Max Gradients}
\label{appendix:monotonicity}

\begin{lemma}
    \label{lemma:monotonicity-min_max}
    Consider a continuously differentiable function $f : \mathbb{R}^n \times \mathbb{R}^m \ni (x,y) \mapsto \mathbb{R}$ such that $f(x, y)$ is $\mu$-strongly convex with respect to $x$ for all $y \in \mathbb{R}^m$, and $f(x,y)$ is $\mu$-strongly concave with respect to $y$ for all $x \in \mathbb{R}^n$. Then, the operator of the min-max gradients is $\mu$-monotone.
\end{lemma}

Before we proceed with the proof of this standard lemma, let us first recall that the operator associated with the min-max gradients is defined as 

\begin{equation}
    F(x,y) := \begin{bmatrix}
    \nabla_x f(x,y) \\
    - \nabla_y f(x,y)
    \end{bmatrix},
\end{equation}
for some continuously differentiable function $f(x,y)$.

\begin{proof}[Proof of \Cref{lemma:monotonicity-min_max}]
Following the techniques of Rockafellar and Rockajellm \cite{Rockafellar_americanmath.}, we obtain that for all $x_1, x_2 \in \mathbb{R}^n$ and $y_1, y_2 \in \mathbb{R}^m$,

\begin{subequations}
\begin{align}
    f(x_2, y_1) \geq f(x_1, y_1) + \langle \nabla_x f(x_1, y_1), x_2 - x_1 \rangle + \frac{\mu}{2} || x_2 - x_1 ||^2; \\
    f(x_1, y_2) \geq f(x_2, y_2) + \langle \nabla_x f(x_2, y_2), x_1 - x_2 \rangle + \frac{\mu}{2} || x_1 - x_2 ||^2; \\
    - f(x_1, y_2) \geq - f(x_1, y_1) + \langle -\nabla_y f(x_1, y_1), y_2 - y_1 \rangle + \frac{\mu}{2} || y_2 - y_1 ||^2; \\
    - f(x_2, y_1) \geq - f(x_2, y_2) + \langle -\nabla_y f(x_2, y_2), y_1 - y_2 \rangle + \frac{\mu}{2} || y_1 - y_2 ||^2.
\end{align}
\end{subequations}
Adding the four inequalities yields that 

\begin{align*}
    \langle x_1 - x_2, \nabla_x f(x_1, y_1) - \nabla_x f(x_2, y_2) \rangle + \langle y_1 - y_2, - \nabla_y f(x_1, y_1) + \nabla_y f(x_2, y_2) \rangle \notag \\
    \geq \mu || x_1 - x_2 ||^2 + \mu || y_1 - y_2||^2,
\end{align*}
concluding the proof.
\end{proof}

\section{Single-Call Variants of the Extra-Gradient Method}
\label{appendix:single_call}

Another application of the framework employed in our work consists of establishing equivalence between different optimization methods. Indeed, in this section we show equivalence between the $\ogd$ method, and other \emph{single-call} variants of the Extra-Gradient method. We stress that it is well-known that this equivalence holds only for the \emph{unconstrained} dynamics \cite{DBLP:conf/nips/HsiehIMM19}.

\subsection{Past Extra-Gradient Descent}

First, we consider Popov's Past Extra-Gradient Descent method ($\pegd$), which boils down to the following update rules:

\begin{subequations}
\label{eq:PEGD}
\begin{align}
    x_{k + 1/2} = x_k - \eta F(x_{k - 1/2}); \\
    x_{k + 1} = x_k - \eta F(x_{k + 1/2}).
\end{align}
\end{subequations}
Thus, the $\pegd$ controller can be expressed as 

\begin{subequations}
\label{eq:PEGD-normal_form}
\begin{align}
    \begin{bmatrix}
    \xi_{k+1}^{(1)} \\
    \xi_{k+1}^{(2)} \\
    \xi_{k+1}^{(3)}
    \end{bmatrix}
    &=
    \begin{bmatrix}
    \mathbf{0}_d & \mathbf{I}_d & -\eta \mathbf{I}_d \\
    \mathbf{0}_d & \mathbf{I}_d & \mathbf{0}_d \\
    \mathbf{0}_d & \mathbf{0}_d & \mathbf{0}_d
    \end{bmatrix}
    \begin{bmatrix}
    \xi_{k}^{(1)} \\
    \xi_{k}^{(2)} \\
    \xi_{k}^{(3)}
    \end{bmatrix}
    +
    \begin{bmatrix}
    \mathbf{0}_d \\
    - \eta \mathbf{I}_d \\
    \mathbf{I}_d
    \end{bmatrix}
    v_k;
    \\
    u_k &=
    \begin{bmatrix}
    \mathbf{0}_d & \mathbf{I}_d & -\eta \mathbf{I}_d
    \end{bmatrix}
    \begin{bmatrix}
    \xi_{k}^{(1)} \\
    \xi_{k}^{(2)} \\
    \xi_{k}^{(3)}
    \end{bmatrix},
\end{align}
\end{subequations}
where $\xi_k^{(1)}, \xi_k^{(2)}, \xi_k^{(3)} \in \mathbb{R}^d$ represent the state variables of the system. To see this, first notice that $u_k = \xi_k^{(2)} - \eta \xi_k^{(3)} = \xi_{k+1}^{(1)}$; hence, $v_k = F(\xi_{k+1}^{(1)} + x^*)$. Moreover, it follows that $ \xi_{k+1}^{(1)} = \xi_k^{(2)} - \eta F(\xi_k^{(1)} + x^*)$ and
$\xi_{k+1}^{(2)} = \xi_k^{(2)} - \eta \nabla (\xi_{k+1}^{(1)} + x^*)$. Thus, substituting $\xi_{k}^{(1)} = x_{k - 1/2} - x^*$ and $\xi_k^{(2)} = x_k - x^*$ gives the PEGD dynamics of \Cref{eq:PEGD}.

\begin{proposition}
    \label{proposition:TF-PEGD}
    The transfer matrix of the Past Extra-Gradient Descent controller can be expressed as $\mathbf{K}(z) = K(z) \mathbf{I}_d$, where $K(z) = \eta (1-2z)/(z^2-z)$.
\end{proposition}

\begin{proof}
Let $(v, u)$ be the input/output pair of the controller. \Cref{eq:PEGD-normal_form} implies that $u_k = \xi_k^{(2)} - \eta v_{k-1} = ( \xi_{k-1}^{(2)} - \eta v_{k-1} ) - \eta v_{k-1} = (u_{k-1} + \eta v_{k-2}) - 2\eta v_{k-1}$. Thus, transferring to the $z$-domain gives us that $U(z) = z^{-1} U(z) + \eta z^{-2} V(z) - 2\eta z^{-1} V(z)$.
\end{proof}

As expected, the transfer function for the $\pegd$ method coincides with that of $\ogd$, implying the equivalence of the two methods.

\subsection{Reflected Gradient Descent}

We also consider the Reflective Gradient Descent method (henceforth $\rgd$) \cite{DBLP:journals/siamjo/Malitsky15}, which can be expressed through the following equations: 

\begin{subequations}
\label{eq:RGD}
\begin{align}
    x_{k + 1/2} = x_k - (x_{k-1} - x_k); \\
    x_{k + 1} = x_k - \eta F(x_{k + 1/2}).
\end{align}
\end{subequations}
Thus, the RGD controller can be expressed as 

\begin{subequations}
\label{eq:RGD-normal_form}
\begin{align}
    \begin{bmatrix}
    \xi_{k+1}^{(1)} \\
    \xi_{k+1}^{(2)} \\
    \xi_{k+1}^{(3)}
    \end{bmatrix}
    &=
    \begin{bmatrix}
    \mathbf{0}_d & 2 \mathbf{I}_d & \mathbf{I}_d \\
    \mathbf{0}_d & \mathbf{I}_d & \mathbf{0}_d \\
    \mathbf{0}_d & \mathbf{I}_d & \mathbf{0}_d
    \end{bmatrix}
    \begin{bmatrix}
    \xi_{k}^{(1)} \\
    \xi_{k}^{(2)} \\
    \xi_{k}^{(3)}
    \end{bmatrix}
    +
    \begin{bmatrix}
    \mathbf{0}_d \\
    - \eta \mathbf{I}_d \\
    \mathbf{0}_d
    \end{bmatrix}
    v_k;
    \\
    u_k &=
    \begin{bmatrix}
    \mathbf{0}_d & 2 \mathbf{I}_d & - \mathbf{I}_d
    \end{bmatrix}
    \begin{bmatrix}
    \xi_{k}^{(1)} \\
    \xi_{k}^{(2)} \\
    \xi_{k}^{(3)}
    \end{bmatrix}.
\end{align}
\end{subequations}

As a result, the following proposition implies that the controller of $\rgd$ coincides with the controller of the other---previously considered---single-call variants of Extra-Gradient: 

\begin{proposition}
    \label{proposition:TF-RGD}
    The transfer function of the Reflective Gradient Descent controller can be expressed as $\mathbf{K}(z) =  K(z) \mathbf{I}_d$, where $K(z) = \eta (1-2z)/(z^2-z)$.
\end{proposition}

\begin{proof}
From \eqref{eq:RGD-normal_form} we obtain that $u_k = 2\xi_k^{(2)} - \xi_{k-1}^{(2)} = 2(\xi_{k-1}^{(2)} - \eta v_{k-1}) - \xi_{k-1}^{(2)} = 2 \xi_{k-1}^{(2)} - \xi_{k-2}^{(2)} + \eta v_{k-2} - 2\eta v_{k-1} = u_{k-1} + \eta v_{k-2} - 2\eta v_{k-1}$, and taking the $Z$-transform concludes the proof.
\end{proof}

\end{document}